\documentclass[12pt]{amsart}

\usepackage[hmargin=0.8in,height=8.9in]{geometry}

\usepackage{amssymb,amsthm, times}
\usepackage{delarray,verbatim}
\usepackage{ifpdf}
\ifpdf
\usepackage[pdftex]{graphicx}
 \else
\usepackage[dvips]{graphicx}
 \fi

\usepackage{bm}

\usepackage{ifpdf}
\usepackage{color}
\definecolor{webgreen}{rgb}{0,.5,0}
\definecolor{webbrown}{rgb}{.8,0,0}
\definecolor{emphcolor}{rgb}{0.95,0.95,0.95}
\usepackage{comment} 

\usepackage{hyperref}
\hypersetup{%
          colorlinks=true,
          linkcolor=webbrown,
          filecolor=webbrown,
          citecolor=webgreen,
          breaklinks=true}
\ifpdf \hypersetup{pdftex,
            pdfstartview=FitH, 
            bookmarksopen=true,
            bookmarksnumbered=true
} \else \hypersetup{dvips} \fi

\newcommand {\B}{\mathcal{B}}

\numberwithin{equation}{section}

\newtheorem{theorem}{Theorem}[section]
\newtheorem{proposition}{Proposition}[section]

\newtheorem{remark}{Remark}[section]
\newtheorem{lemma}{Lemma}[section]

\newtheorem{assump}{Assumption}[section]

\numberwithin{remark}{section} \numberwithin{proposition}{section}
\numberwithin{corollary}{section}
\newcommand {\R}{\mathbb{R}}

\newcommand {\p}{\mathbb{P}}

\newcommand {\E}{\mathbb{E}}

\newcommand{\diff}{{\rm d}}

\newcommand{\lev}{L\'{e}vy }

\newcommand{\e}{\mathbb{E}}

\begin{document}
\title[Optimal Periodic replenishment policies]{Optimal periodic replenishment policies for spectrally positive L\'evy demand processes} 
\thanks{This version: \today.   }

\author[J. L. P\'erez]{Jos\'e-Luis P\'erez$^\dag$}

\thanks{$\dagger$\, Department of Probability and Statistics, Centro de Investigaci\'on en Matem\'aticas A.C. Calle Jalisco s/n. C.P. 36240, Guanajuato, Mexico.   }
\author[K. Yamazaki]{Kazutoshi Yamazaki$^\ddagger$}
\thanks{$\ddagger$\, Department of Mathematics,
Faculty of Engineering Science, Kansai University, 3-3-35 Yamate-cho, Suita-shi, Osaka 564-8680, Japan.   }

\author[A. Bensoussan]{Alain Bensoussan$^*$}
\thanks{$*$\, Naveen Jindal School of Management, University of Texas at Dallas, 800 W Campbell Road,
Richardson, TX 75080, USA}
\thanks{The authors thank the anonymous referees and associate editor for careful reading of
the paper and constructive comments and suggestions. This paper was supported  by MEXT KAKENHI grant no.
17K05377, 19H01791 and 20K03758. }
\date{}

\begin{abstract} 
We consider a version of the stochastic inventory control problem for a spectrally positive \lev demand process, in which the inventory can only
be replenished at independent exponential times. 
We show the optimality of a periodic barrier replenishment policy
that restocks any shortage below a certain threshold at each replenishment opportunity.
The optimal policies and value functions are concisely written in terms of the scale functions. Numerical results are also provided. 
\\
\noindent \small{\noindent  AMS 2020 Subject Classifications: 60G51, 93E20, 90B05 }
\end{abstract}

%
\maketitle


\section{Introduction} 

The classical continuous-time inventory model aims to
optimally control the inventory level so as to strike a balance between minimizing the inventory costs and replenishment costs. The inventory in the absence of control is typically assumed to follow a Brownian motion, a compound Poisson process, or a mixture of the two. Under the assumption that the inventory can be monitored continuously and replenishment can be made instantaneously, the existing results have
shown
the optimality of a \emph{barrier} or an \emph{$(s,S)$}-policy, depending on whether fixed (replenishment)  costs are considered. For a comprehensive review and various inventory models, see \cite{Bensoussan_book}.

In this study, we consider a new extension of the inventory model under the constraint that replenishment 
opportunities occur at the arrival times of an independent Poisson process. 
This is because, in reality, one can monitor the inventory only at intervals and, hence, barrier or $(s, S)$ policies are difficult to implement in practice. Recently, similar extensions have been studied in the context of insurance applications \cite{ATW, NobPerYamYan, PerYam}.


Analytical solutions can be pursued under the assumption of Poissonian replenishment opportunities in which, thanks to the memoryless property, the waiting time until the next opportunity is always (conditionally) exponentially distributed. With other replenishment opportunity times, the state space must be expanded to make the problem Markovian, and, to our knowledge, one must resort to numerical approaches rather than analytical solutions.

One important motivation for considering the Poissonian interarrival model is its potential applications in approximating the constant interarrival time cases. 
In the mathematical finance literature, randomization techniques (see, e.g., \cite{Carr}) are known as efficient in approximating constant maturity problems with those with Erlang-distributed maturities. 
In particular, for short maturity cases, it is known empirically that accurate approximations can be obtained by simply replacing the constant with exponential random variables \cite{leung2015analytic}.

Although the Poissonian assumption simplifies the considered problem,  it is still significantly more challenging and interesting in comparison to the continuous monitoring case.  The solutions depend directly on the rate of Poisson arrivals, and it is, therefore, of interest to study its sensitivity.




In this study, we focus on the discounted continuous-time model driven by a spectrally positive \lev demand process. In other words, the inventory, in the absence of control, follows a \lev process with only negative jumps. As is typically assumed in the literature, the inventory cost is modeled by a convex function, and the cost of replenishment is assumed to be proportional to the order amount.  
Under these assumptions, the classical continuous monitoring case admits a simple solution (see Section 7 of \cite{Y}): it is optimal to reflect the inventory process at a suitably chosen barrier, and the value function is expressed concisely in terms of the so-called \emph{scale function} (see also \cite{BLS} and Sections 4-6 of \cite{Y} for the cases with fixed costs).


This study aims to show the optimality of a \emph{periodic barrier replenishment policy}, which restocks any shortage below a certain threshold at each replenishment opportunity. The corresponding controlled inventory process becomes the \emph{Parisian reflected process} studied in \cite{APY, PerYam2}.
 We show that a periodic barrier replenishment policy is indeed optimal over the set of all admissible policies.

We follow the classical \emph{guess-and-verify} procedure to solve this stochastic control problem: 

\begin{enumerate}

\item The first step is to compute the expected net present value (NPV) of replenishment and inventory costs under periodic barrier replenishment policies. Replenishment costs, which are the expected amount of total discounted Parisian reflection, have been computed in \cite{APY}.  Inventory costs require the resolvent identity, which we compute using a similar method as in \cite{APY}. These admit semi-explicit expressions written in terms of the scale function.

\item In the second step, we select the optimal periodic barrier, which we call $b^*$ in the current study.  
We choose its value so that the slope of the candidate value function at the barrier equals the negative of the unit replenishment cost.

\item In the final step, we confirm the optimality of the selected candidate optimal policy.  To this end, we obtain a verification lemma (sufficient condition for optimality), 
which requires the value function to be sufficiently smooth and satisfy certain variational inequality. 
 By taking advantage of the existing analytical properties of the scale function, as well as some fluctuation identities, 
 we confirm that the candidate value function indeed satisfies these conditions.

 \end{enumerate}

One major advantage of applying these three steps is that one can solve the problem for a general spectrally positive \lev demand process (of both bounded and unbounded variations) without specifying a particular type of \lev measure. By reducing the problem to certain analyses on the scale function of the underlying \lev process, we avoid the use of integro-differential equation techniques, which tend to be difficult, particularly when the \lev measure has infinite activity.

%
%

The rest of the paper is organized as follows.  In Section \ref{dividends-policy}, we model the problem considered. 
Section \ref{sec_verification_lemma} gives the verification lemma.  In Section \ref{section_periodic_policy}, we study the periodic barrier replenishment policy and compute the corresponding expected NPV of the total costs.  In Section \ref{section_selection_barrier}, we select the candidate barrier. In Section \ref{section_optimality}, the optimality of the selected policy is shown and confirmed numerically. Long proofs and technical results are deferred to the appendix. Throughout the paper, superscripts $x^+ := \max (x, 0)$ and $x^- := \max(-x, 0)$ are used to indicate the positive and negative parts of $x$, respectively. The left and right hand limits are written as $f(x-):= \lim_{y \uparrow x} f(y)$ and $f(x+):= \lim_{y \downarrow x} f(y)$, respectively, whenever they exist.


\section{Inventory models with periodic replenishment opportunities}
\label{dividends-policy} 

Let $(\Omega, \mathcal{F}, \mathbb{P})$ be a probability space on which a stochastic process $D = ( D(t); t \geq 0)$ with $D(0) = 0$, modeling the aggregate demand of a single item, is defined.  Under the conditional probability $\mathbb{P}_x$, for $x \in \R$, the initial level of inventory is given by $x$ (in particular, we let $\mathbb{P} \equiv \mathbb{P}_0$).  Hence, the inventory, in the absence of control, follows the stochastic process
\begin{align*}
X(t) := x - D(t), \quad t \geq 0. 
\end{align*}

We consider a scenario where the item can be replenished only at the arrival times $\mathcal{T}_r :=(T(i); i\geq 0 )$ of a Poisson process $N^r=( N^r(t); t\geq 0) $ with intensity $r>0$, which is independent of  $X$ (and $D$).  
In other words, the interarrival times $T(i)-T(i-1)$, $i\geq 1$ (with $T(0) := 0$) are independent and exponentially distributed with mean $1/r$.  Let $\mathbb{F} := (\mathcal{F}(t); t \geq 0)$ be the filtration generated by the process $(X, N^r)$.

\par In this setting, an admissible policy, representing the cumulative amount of replenishment  $\pi := \left( R^{\pi}(t); t \geq 0 \right)$ is a nondecreasing, right-continuous, and $\mathbb{F}$-adapted process such that
\[
R^{\pi}(t)=\int_{[0,t]}\nu^{\pi}(s)\diff N^r(s),\qquad\text{$t\geq0$,}
\]
for a c\'agl\'ad process $\nu^\pi$.
In particular, the replenishment at the $i$-th replenishment opportunity $T(i)$  is given by $\nu^\pi(T(i))$ for each $i\geq 1$. 
The controlled inventory  process $U^\pi$ becomes
\[
U^\pi(t) := X(t)+ R^{\pi} (t) =X(t)+\sum_{i=1}^{\infty}\nu^{\pi}(T(i))1_{\{T(i)\leq t\}}, \qquad\text{$t \geq 0$}. 
\]


We fix a discount factor $q > 0$ and a unit cost/reward of controlling $C \in \R$. Associated with the policy $\pi \in \mathcal{A}$, the cost of inventory is modeled by $\int_0^{\infty} e^{-q t}f(U^{\pi}(t)) \diff t$ for a measurable function $f: \R \to \R$ and that of controlling is given by $C \int_{[0,\infty)} e^{-q t} \diff R^{\pi}(t)$.
The problem is to minimize their expected sum
\begin{align*}
v_{\pi} (x) := \mathbb{E}_x \left[ \int_0^{\infty} e^{-q t}f(U^{\pi}(t)) \diff t + C \int_{[0,\infty)} e^{-q t} \diff R^{\pi}(t) \right], \quad x \in \R,
\end{align*}
over the set of all admissible policies $\mathcal{A}$ that satisfy all the constraints described above and
\begin{align}
\E_x \left[ \int_{[0,\infty)} e^{-q t} \diff R^{\pi}(t) \right] < \infty. \label{assump_R_finite}
\end{align}
The problem is to compute the value function
\begin{equation}
v(x):=\inf_{\pi \in \mathcal{A}}v_{\pi}(x), \quad x \in \R, \label{the_prob}
\end{equation}
and to obtain the optimal policy $\pi^*$ that attains it, if such a policy exists.

\subsection{Spectrally one-sided \lev processes} 
We shall consider the case where the demand $D$ follows a spectrally positive \lev process, or equivalently $X$ is a spectrally negative \lev process.  We exclude the case $X$ is the negative of a subordinator so that it does not have  monotone paths a.s.
We denote the Laplace exponent of $X$ by
$\kappa: [0, \infty) \to \R$ such that
 $\E [e^{\theta X(t)}]=e^{ t \kappa(\theta)}$ for $t, \theta \geq 0$, 
with its L\'evy-Khintchine decomposition
 $$ \kappa(\theta) = \frac{\sigma^2}{2}\theta^2 + \gamma \theta + \int_{(-\infty, 0)}[e^{\theta y} - 1 - \theta y1_{\{y > -1\}} ]\Pi(\diff y), \quad \theta \geq 0.$$
Here, $\sigma \geq 0,$ $\gamma \in \R$, and the \lev measure $\Pi$ satisfies $ \int_{(-\infty, 0)} (1\wedge y^2) \Pi(\diff y) < \infty$. 

It is known (see, e.g., Lemma 2.12 of \cite{K}) that $X$ has paths of bounded variation if and only if $\sigma=0$ and $\int_{(-1, 0)} |y|\Pi(\mathrm{d}y) < \infty$. 
For the bounded variation case,
$X$ can be written as 
\begin{align*}
X(t)=ct-S(t), \quad t\geq 0, \quad \textrm{where } \quad
c:=\gamma-\int_{(-1,0)} y\Pi(\mathrm{d}y), 
\end{align*}
 and $(S(t); t\geq0)$ is a driftless subordinator. 
Here, by the assumption that $X$ is not the negative of a subordinator, necessarily we have $c>0$.


\subsection{Assumptions} 
We solve the problem \eqref{the_prob} under the following standing assumptions on
the \lev process $X$ and the running cost function $f$.
	\begin{assump} \label{assump_levy_measure}
			We assume that there exists $\bar{\theta} > 0$ such that  $\int_{(-\infty, -1]} \exp (\bar{\theta} |z|) \Pi(\diff z) < \infty$. This guarantees that $\E [X(1)] = \kappa'(0+) > -\infty$.
	\end{assump}
	\begin{assump} \label{assump_f}
		\begin{itemize}
		\item[(i)] We assume that $f$ is convex and has at most polynomial growth in the tail.  That is to say, there exist $k_1, k_2, m > 0$ and $N \in \mathbb{N}$ such that $|f(x)| \leq k_1 + k_2 |x|^N$ for all $x \in \R$ such that $|x| > m$.
		\item[(ii)] We assume that $\displaystyle f'(-\infty)<- Cq < f'(\infty)$ where $f'(\infty) := \lim_{x \rightarrow \infty}f'(x) \in (-\infty, \infty]$ and $f'(-\infty) := \lim_{x \rightarrow -\infty}f'(x) \in [-\infty, \infty)$.
		\end{itemize}
	\end{assump}
	
	These assumptions are critical for our analysis, and similar assumptions are imposed in the existing literature (see, e.g., \cite{BLS, HPY}). 
	

\begin{remark} \label{remark_finiteness_resolvents}
 By Assumptions \ref{assump_levy_measure} and \ref{assump_f} we have that $\E_x\left[ \int_0^{\infty} e^{-qt} |f (X(t))| \diff t \right] < \infty$ for all $x \in \R$.   For its proof, see the proof of Lemma 7.5 of \cite{Y}.
\end{remark}

\section{Verification Lemma} \label{sec_verification_lemma}


We first obtain
 the verification lemma for the considered problem.
Throughout the paper, we call a measurable function $g$ \emph{sufficiently smooth} on $\R$ if $g$ is $C^1 (\R)$ (resp.\ $C^2 (\R)$) when $X$ has paths of bounded (resp.\ unbounded) variation.
Let $\mathcal{L}$ be the operator acting on a sufficiently smooth function $g$, defined by
\begin{equation*}
\begin{split}
\mathcal{L} g(x)&:= \gamma g'(x)+\frac{\sigma^2}{2}g''(x) +\int_{(-\infty,0)}[g(x + z)-g(x)-g'(x)z\mathbf{1}_{\{-1 < z < 0\}}]\Pi(\mathrm{d}z). 
\end{split}
\end{equation*}
Also, we define the operator $\mathcal{M}$ acting on a measurable function $g$,
\begin{align}
	\mathcal{M} g(x) := \inf_{l \geq 0}\{C l+ g(x+l)\}. \label{def_M_operator}
\end{align}	
\begin{lemma}[Verification lemma]
	\label{verificationlemma}
	Suppose $\hat{\pi} \in \mathcal{A}$ is such that $w:=v_{\hat{\pi}}$ is sufficiently smooth on $\R$, 
	has polynomial growth (see Assumption \ref{assump_f}), and satisfies
	\begin{align}
	\label{HJB-inequality}
	(\mathcal{L} - q)w(x)+r (\mathcal{M} w(x) - w(x))+f(x) = 0,  \quad x \in\R.
	\end{align} 
	Then $v(x)=w(x)$ for all $x \in \R$ and hence $\hat{\pi}$ is an optimal policy. 
\end{lemma}


\begin{remark} 
\label{remark_verificatin_lemma}
(1) The equality \eqref{HJB-inequality} can be intuitively explained by the Bellman's principle.
For a small time interval $\Delta_t$, the corresponding Bellman's equation is expected to be approximated as
\begin{align*}
	v(x) = e^{-r  \Delta_t}\E_x [ e^{-q  \Delta_t}v(X(\Delta_t))]  + (1-e^{-r \Delta_t}) \E_x [  e^{-q  \Delta_t} \mathcal{M}  v(X(\Delta_t)]+\E_x\left[\int_0^{\Delta_t} e^{-qs} f(X(s)) \diff s\right] + o(\Delta_t),
\end{align*}	
 where $e^{-r\Delta_t}$ is the probability of no replenishment opportunities over $(0,\Delta_t)$, and $1-e^{-r\Delta_t}$  its complement. Hence, using It\^o's formula, by dividing by $\Delta_t$ and taking $\Delta_t \downarrow 0$, we arrive at \eqref{HJB-inequality}. 

(2) Define the set
$\mathcal{C} := \{ x \in \R: (\mathcal{L} - q)v(x) + f(x) = 0 \}$. 
Then, $\mathcal{C}$ can be understood as the continuation region, and $\mathcal{D} := \R \backslash \mathcal{C}$ as the control region at which replenishment is made whenever the replenishment opportunity arrives.

In this paper, we aim to show that $\mathcal{C} = [b^*, \infty)$ and $\mathcal{D} = (-\infty,b^*)$ for some $b^* \in \R$.  This property is closely related to the convexity of $v$ and its slope at $b^*$. To see this, if $v$ is convex and $v'(b^*) = -C$, then necessarily we have
$\mathcal{M} v(x) - v(x) = 0$ if and only if $x \geq b^*$.

(3) There are both similarities and differences  with the classical singular control case and the version where the control process must be absolutely continuous with a  bounded density
(see (4.2) of \cite{HPY}). While the forms of the variational inequalities differ,  the convexity and the slope condition at the candidate barrier are the key elements needed  as in the current paper.
\end{remark}

\begin{proof}[Proof of Lemma \ref{verificationlemma}]
By the definition of $v$ as an infimum, it follows that $w(x) \geq v(x)$ for all $x \in \R$. Hence, it suffices to show the opposite inequality.

Fix $x \in \R$ and  $\pi\in \mathcal{A}$ with its corresponding inventory process $U^\pi$. 
Let $(T_n)_{n\in\mathbb{N}}$ be 
defined by $T_n :=\inf\{t>0:|U^\pi (t)| > n \}$; 
here and throughout, let $\inf \varnothing = \infty$. 

Because $U^\pi$ is a semi-martingale and $w$ is sufficiently smooth on $\R$, 
the change of variables/It\^o's formula 
(see Theorems II.31 and II.32 of \cite{protter}) gives under $\mathbb{P}_x$ that 
\begin{align*}
\mathrm{e}^{-q(t\wedge T_n)}w(U^\pi(t\wedge T_n))&-w(x)
=  -\int_{0}^{t\wedge T_n}\mathrm{e}^{-qs} q w(U^\pi(s-)) \mathrm{d}s
+\int_{[0, t\wedge T_n]}\mathrm{e}^{-qs}w'(U^\pi(s-)) \mathrm{d} X(s)  \\
&+ \frac{\sigma^2}{2}\int_0^{t\wedge T_n}\mathrm{e}^{-qs}w''(U^\pi(s-))\mathrm{d}s+\sum_{0 \leq s\leq t\wedge T_n}\mathrm{e}^{-qs}\big[\Delta w\big(U^\pi(s-)+\nu^{\pi}(s)\big)\Delta N^{r}(s) \big] \\
& + \sum_{0 \leq s\leq t\wedge T_n}\mathrm{e}^{-qs} \big[\Delta w\big(U^\pi(s-)+\Delta X(s)\big)-w'(U^\pi(s-))  \Delta X(s)  \big] \\
&=   \int_{0}^{t\wedge T_n}\mathrm{e}^{-qs}   (\mathcal{L}-q)w(U^\pi(s-))   \mathrm{d}s
- C \int_{[0,t\wedge T_n]}\mathrm{e}^{-qs}\nu^{\pi}(s)\mathrm{d}N^{r}(s)  \\
&+\int_0^{t\wedge T_n}\mathrm{e}^{-qs}r\left[ C \nu^\pi(s)+w\big(U^\pi(s-)+\nu^{\pi}(s)\big)-w(U^\pi(s-))\right] \mathrm{d}s  + M(t \wedge T_n)
\end{align*}
where we define, for $t \geq 0$, with $\tilde{\mathcal{N}}(\diff s\times \diff y) := \mathcal{N}(\diff s\times \diff y)-\Pi(\diff y) \diff s$,
\begin{align*}
\begin{split}
&M(t\wedge T_n) := \int_0^{t\wedge T_n} \sigma  \mathrm{e}^{-qs} w'(U^{\pi}(s-)) \diff B(s) +\lim_{\varepsilon\downarrow 0} \int_{[0,t\wedge T_n]} \int_{(-1, -\varepsilon)}  \mathrm{e}^{-qs}w'(U^{\pi}(s-))y \tilde{\mathcal{N}}(\diff s\times \diff y)\\
&+\int_{[0,t\wedge T_n]} \int_{(-\infty,0)} \mathrm{e}^{-qs} \big[ w(U^\pi(s-)+y)-w(U^\pi(s-))-w'(U^\pi(s-))y\mathbf{1}_{\{y\in (0, 1)\}} \big] \tilde{\mathcal{N}}(\diff s\times \diff y) \\
&+\int_{[0,t\wedge T_n]} \mathrm{e}^{-qs}\left[C \nu^\pi(s)+w\big(U^\pi(s-)+\nu^{\pi}(s)\big)-w(U^\pi(s-))\right] \mathrm{d} (N^r(s)-rs). 
\end{split}
\end{align*}
Here, $( B(s); s \geq 0 )$ is  a standard Brownian motion and $\mathcal{N}$ 
is a Poisson random measure in   the measure space  $([0,\infty)\times (-\infty, 0),\B [0,\infty)\times \B (-\infty,0), \diff s \times \Pi( \diff x))$. 
By the definition of $\mathcal{M}$ as in \eqref{def_M_operator},
\begin{equation*}
\begin{split}
w(x) \leq &
-\int_{0}^{t\wedge T_n}\mathrm{e}^{-qs}  \Big[ (\mathcal{L}-q)w(U^\pi(s-))+r\big(\mathcal{M} w(U^\pi(s-))-w(U^\pi(s-))\big) \Big]  \mathrm{d}s\\
& + C \int_{[0, t\wedge T_n]}\mathrm{e}^{-qs}\nu^{\pi}(s)\mathrm{d}N^{r}(s) - M(t\wedge T_n) + \mathrm{e}^{-q(t\wedge T_n)}w(U^\pi(t\wedge T_n)).
\end{split}
\end{equation*}
	Using  the assumption \eqref{HJB-inequality}, together with the fact that the process $(M(t \wedge T_n); t\geq0 )$ is a zero-mean $\mathbb{P}_x$-martingale (see Corollary 4.6 of \cite{K}),  after taking expectations, we obtain
\begin{equation}\label{aux_2}
	w(x) \leq \mathbb{E}_x \left[\int_0^{t\wedge T_n}\mathrm{e}^{-qs}f(U^\pi(s)) \diff s + C \int_{[0, t\wedge T_n]}\mathrm{e}^{-qs}\nu^{\pi}(s)\mathrm{d}N^{r}(s) + \mathrm{e}^{-q(t\wedge T_n)}w(U^\pi(t\wedge T_n)) \right].
\end{equation}	

We shall now take $t,n \uparrow \infty$ in the above inequality to complete the proof.
First, 
assumption \eqref{HJB-inequality} and the fact that $\mathcal{M} w \leq w$ imply that $(\mathcal{L} - q)w(y) +f(y)\geq 0$ for $y\in\R$. Because $w$ is sufficiently smooth 
and is of polynomial growth, by It\^o's formula together with dominated convergence, we have 
	$w(x) \leq \E_x [\int_0^\infty e^{-qs} f(X(s)) \diff s ]$ for all $x \in \R$ (for more details, see the proof of Lemma 7.5 of \cite{Y}). 
This, together with the strong Markov property, implies
\begin{align}\label{aux_1}
	\E_x\left[\mathrm{e}^{-q(t\wedge T_n)}w(U^\pi(t\wedge T_n))\right] \leq \E_x \Big[ \int_{t \wedge T_n}^\infty e^{-qs} f \big(R^\pi(t\wedge T_n) + X(s) \big) \diff s \Big].
\end{align}
Now, following the same steps as the proof of Theorem 7.1 of \cite{Y}, we have	
\begin{align*}
\E_x \Big[ \int_{t \wedge T_n}^\infty e^{-qs} f \big(R^\pi(t\wedge T_n) + X(s) \big) \diff s \Big] &\leq \E_x \Big[ \int_{[t \wedge T_n, \infty)} e^{-qs} \Big(f(U^\pi(s))\diff s + C \diff R^\pi(s) \Big)\Big]\\& + \E_x \Big[ \int_{t \wedge T_n}^\infty e^{-qs} \Big(f(X(s))+CqX(s) \Big)\diff s \Big].
\end{align*}
By using this and \eqref{aux_1} in \eqref{aux_2}, we obtain 
	$w(x) \leq v_\pi(x)
	+ \E_x [ \int_{t \wedge T_n}^\infty e^{-qs} (f(X(s))+CqX(s))\diff s ]$. 
Because $ \E_x [ \int_{t \wedge T_n}^\infty e^{-qs} | f(X(s))+CqX(s) | \diff s ] < \infty$ (which holds by Remark \ref{remark_finiteness_resolvents}), upon taking $t, n \uparrow \infty$ via monotone convergence,  we have $w(x) \leq v_\pi(x)$, as desired.
\end{proof}

\section{Periodic barrier replenishment policies} \label{section_periodic_policy}

The objective of this paper is to show the optimality of the \emph{periodic barrier replenishment policy} $\pi^b$, $b \in \R$, that pushes the inventory up to $b$ at the observation times $\mathcal{T}_r$ whenever it is below $b$. The resulting inventory process is precisely the \emph{Parisian reflected \lev process} of \cite{APY}.  

We denote, by $R_r^b$ and $U_r^b$, the aggregate sum of replenishment and the resulting inventory, respectively. 
More concretely, we have
\begin{align*} 
U_r^b(t) = X(t) \quad \textrm{and} \quad R_r^b(t) = 0, \quad 0 \leq t < T_b^-(1)
\end{align*}
where $T_{b}^-(1) := \inf\{  S \in \mathcal{T}_r :  X(S-) < b \}$ is the first replenishment time.
 The inventory is then pushed up by the amount $\Delta R_r^b(T_b^-(1)) = b-X(T_b^-(1) -)$ so that $U_r^b(T_b^-(1)) = b$. For $T_b^-(1) \leq t < T_b^-(2)  := \inf\{ S \in \mathcal{T}_r : S > T_b^-(1), U_r^b(S-) < b \}$, we have $U_r^b(t) = X(t) + (b-X(T_b^-(1)-))$ and $R_r^b(t) = R_r^b(T_b^-(1))$.  The controlled inventory process can be constructed by repeating this procedure.

We have the following decomposition:
\begin{align*}
U_r^b(t) = X(t) + R_r^b(t), \quad t \geq 0,
\end{align*}
 with
 \begin{align*}
R_r^b(t) = 
\sum_{i=1}^\infty (b-U_r^b(T_b^-(i)-)) 1_{\{ T_b^-(i) \leq t \}} =\int_{[0,t]} (b-U_r^b(s-))^+ \diff N^r(s),
\quad t \geq 0, 
\end{align*}
where the replenishment times $(T_{b}^-(n); n \geq 1)$ can be constructed inductively by 
$T_{b}^-(1)$ defined above
and
$T_{b}^-(n+1) := \inf\{ S \in \mathcal{T}_r : S > T_b^-(n), U_r^b(S-) < b \}$ for $n \geq 1$. 
We will see by \eqref{controlling_cost} that the policy $\pi^b:=(R_r^b(t);t\geq0)$ satisfies \eqref{assump_R_finite}, and is hence admissible.

In this section, we compute, via the scale function, the expected NPV of the total costs under $\pi^b$:
\begin{align} \label{v_b}
v_{b} (x) := \mathbb{E}_x \Big[ \int_0^{\infty} e^{-q t}f(U_r^{b}(t)) \diff t + C \int_{[0,\infty)} e^{-q t} \diff R_r^b(t) \Big], \quad  b, x \in \R.
\end{align}

\subsection{Scale functions}
We fix $q,r > 0$. The scale function $W^{(q)}: \R \to [0, \infty)$ of $X$ takes zero on $(-\infty,0)$, and on $[0, \infty)$ it is a strictly increasing function, defined by its Laplace transform:
\begin{align} \label{scale_function_laplace}
\begin{split}
\int_0^\infty  \mathrm{e}^{-\theta x} {W^{(q)}}(x) \diff x &= \frac 1 {\kappa(\theta)-q}, \quad \theta > \Phi(q)  := \sup \{ \lambda \geq 0: \kappa(\lambda) = q\}. \\
\end{split}
\end{align}
In addition, let, for $x \in \R$, 
\begin{align*}
\overline{W}^{(q)}(x) :=  \int_0^x {W^{(q)}}(y) \diff y, \quad 
Z^{(q)}(x) := 1 + q \overline{W}^{(q)}(x), \quad
\overline{Z}^{(q)}(x) := \int_0^x Z^{(q)} (z) \diff z. 
\end{align*}
Note that, for $x \leq 0$, $\overline{W}^{(q)}(x) = 0$, $Z^{(q)}(x) = 1$,  and $\overline{Z}^{(q)}(x) = x$. 
We also define, for $\theta\geq 0$ 
and $x \in \R$, 
\begin{align}
	Z^{(q)}(x, \theta) &:=e^{\theta x} \left( 1 + (q- \kappa(\theta)) \int_0^{x} e^{-\theta z} W^{(q)}(z) \diff z	\right). 
	\label{Z2}
\end{align}
In particular, for $x \in \R$, $Z^{(q)}(x, 0) =Z^{(q)}(x)$ 
and
\begin{align} \label{Z_special}
	\begin{split}
		Z^{(q)}(x, \Phi(q+r)) &=e^{\Phi(q+r) x} \left( 1 -r \int_0^{x} e^{-\Phi(q+r) z} W^{(q)}(z) \diff z \right), \\ 
		Z^{(q+r)}(x, \Phi(q)) &=e^{\Phi(q) x} \left( 1 +r \int_0^{x} e^{-\Phi(q) z} W^{(q+r)}(z) \diff z \right).
		\end{split} \end{align}
Finally, let
		\begin{align*}
		Z^{(q,r)}(x) &:=\frac{r}{q+r} Z^{(q)}(x)
		+\frac{q}{q+r} Z^{(q)}(x,\Phi(q+r) ), \quad x \in \R,
\end{align*}
and, for all $x, y \in \R$, 
\begin{align} \label{W_a_def}
\begin{split}
		W_y^{(q,r)}(x) &:=W^{(q+r)}(x-y)- r\int_{0}^{x}  W^{(q)}(x-z)W^{(q+r)}(z-y)\diff z \\ &=W^{(q)}(x-y)+r\int_0^{-y}W^{(q)}(x-u-y)W^{(q+r)}(u) \diff u, 
		\end{split}
\end{align}
where the second equality holds by (7) of \cite{LRZ}, 
and in particular
$W_y^{(q,r)}(x)=W^{(q)}(x-y)$ for $y \geq 0$. 
	

For the rest of this subsection, we list several fluctuation identities which we use later in the paper. For the spectrally negative \lev process $X$, define
\begin{align*}
\tau_a^- := \inf \left\{ t > 0: X(t) < a \right\} \quad \textrm{and} \quad \tau_a^+ := \inf \left\{ t > 0: X(t) >  a \right\}, \quad a \in \R.
\end{align*}
By using identity (3.19) in \cite{APP2007}, for $x \in \R$ and $\theta \geq 0$,
\begin{align} \label{H_theta}
H^{(q+r)}(x, \theta) := \E_{x}\left[e^{-(q+r)\tau_0^-+ \theta  X(\tau_0^-)} 1_{\{\tau_0^-<\infty\}}\right] = Z^{(q+r)}(x,\theta)-\frac{\kappa(\theta)-(q+r)}{\theta-\Phi(q+r)}W^{(q+r)}(x),
\end{align}
where, in particular,
\begin{align} \label{H_simple}
\begin{split}
H^{(q+r)}(x, \Phi(q)) &= \E_{x}\left[e^{-(q+r)\tau_0^-+\Phi(q)X(\tau_0^-)} 1_{\{\tau_0^-<\infty\}}\right]=Z^{(q+r)} (x, \Phi(q)) - \frac {r W^{(q+r)}(x)} {\Phi(q+r) - \Phi(q)}, \\
H^{(q+r)}(x)&:=H^{(q+r)}(x, 0) = \E_{x}\left[e^{-(q+r)\tau_0^-} \right]=Z^{(q+r)}(x)-\frac{q+r}{\Phi(q+r)}W^{(q+r)}(x). 
\end{split}
\end{align}

For any Borel set $A\subset(-\infty,0]$ and $x\leq0$, by Theorem 2.7(ii) in \cite{KKR}, 
		\begin{equation}
		\E_x\Big[\int_0^{\tau_0^+}e^{-(q+r)t}1_{\{X(t)\in A\}}\diff t \Big]=\int_{A}\Theta^{(q+r)}(x,y)\diff y, \label{killed_resolvent}
		\end{equation}
where we define, for $x,y \in \R$, 
	\begin{align}
	\Theta^{(q+r)}(x,y):= e^{\Phi(q+r)x} W^{(q+r)}(-y) - W^{(q+r)}(x-y). \label{Theta_def}
	\end{align}
\begin{remark} \label{remark_theta_sign}
	(i) For $x, y \leq 0$,  by the identity \eqref{killed_resolvent}, $\Theta^{(q+r)}(x,y) \geq 0$.
	
	(ii) On the other hand, for $x > 0$ and $y \leq x$, $\Theta^{(q+r)}(x,y) \leq 0$. Indeed,
	by \eqref{killed_resolvent},
		\begin{align} \label{f_and_theta}
			\begin{split}
				0\leq \E \Big[\int_{0}^{\tau_x^+} e^{-(q+r)t}1_{\{X(t) \in \diff y \}}\diff t \Big]
				=- e^{-\Phi(q+r)x}\Theta^{(q+r)}(x,y)\diff y.
			\end{split}
		\end{align}
	\end{remark}
Let $\underline{X}$ be the running infimum process of $X$ and $\mathbf{e}_{q+r}$ be an independent exponential random variable with parameter $q+r$.
By Corollary 2.2 of \cite{KKR}, for Borel subsets on $[0, \infty)$, 
\begin{align}
	\p \left( -\underline{X} (\mathbf{e}_{q+r}) \in \diff y \right) = \frac {q+r} {\Phi(q+r)} W^{(q+r)} (\diff y) - (q+r) W^{(q+r)} (y) \diff y,
	\label{density_running_min} 
	\end{align}
where $W^{(q+r)}(\diff y)$ is the measure such that $W^{(q+r)}(y) = \int_{[0,y]}W^{(q+r)}(\diff z)$  (see  \cite[(8.20)]{K}). 
 \begin{remark}\label{remark_smoothness_zero} 
\begin{enumerate}
\item By (8.26) of \cite{K}, the left- and right-hand derivatives of $W^{(q)}$ always exist on $\R \backslash \{0\}$.  In addition, as in, e.g., \cite[Theorem 3]{Chan2011}, if $X$ is of unbounded variation or the \lev measure is atomless, we have ${W^{(q)}} \in C^1(\R \backslash \{0\})$.

\item As in Lemmas 3.1 and 3.2 of \cite{KKR},
\begin{align*} 
\begin{split}
{W^{(q)}} (0) &= \left\{ \begin{array}{ll} 0 & \textrm{if $X$ is of unbounded
variation,} \\ \frac 1 {c} & \textrm{if $X$ is of bounded variation,}
\end{array} \right. \\
{W^{(q)\prime}} (0+) &
=
\left\{ \begin{array}{ll}  \frac 2 {\sigma^2} & \textrm{if }\sigma > 0, \\
\infty & \textrm{if }\sigma = 0 \; \textrm{and} \; \Pi(-\infty,0) = \infty, \\
\frac {q + \Pi(-\infty, 0)} {c^2} &  \textrm{if }\sigma = 0 \; \textrm{and} \; \Pi(-\infty, 0) < \infty.
\end{array} \right.
\end{split}
\end{align*}
\item 
As in Lemma 3.3 of \cite{KKR}, 
$W_{\Phi(q)}(x) := e^{-\Phi(q) x}{W^{(q)}} (x) \nearrow \kappa'(\Phi(q))^{-1}$, as $x \uparrow \infty$.
\end{enumerate}
\end{remark}

\subsection{The computation of $v_b$}


We shall now write the expected NPV of total costs $v_b$ as in \eqref{v_b}.  For the controlling cost, it has already been obtained in Corollary 3.2(iii) of \cite{APY} that, for $b,x \in \R$,
\begin{align} \label{controlling_cost}
\mathbb{E}_x \left[ \int_{[0,\infty)} e^{-q t} \diff R_r^b(t) \right] =  \frac  {\Phi(q+r)- \Phi(q)} {\Phi(q+r) \Phi(q)}  Z^{(q,r)}(x-b) - \frac r {q+r} \Big\{ \overline{Z}^{(q)} (x-b) + \frac {\kappa'(0+)} q \Big\}. 
\end{align}
%
%
Hence, it is left to compute the expected NPV of the inventory cost. 

Recall $H^{(q+r)}$ as in \eqref{H_theta}, and in order to obtain a concise expression 
	for $v_b$ let us define, for $x, y \in \R$, 
\begin{align}
\Upsilon(x,y)
&:=-\Theta^{(q+r)}(x,y)  + r \int_0^x W^{(q)}(x-z) \Theta^{(q+r)}(z,y)  \diff z  \label{def_Upsilon} \\
\label{fun_ups}
	&= W_{y}^{(q,r)}(x) - Z^{(q)}(x, \Phi(q+r)) W^{(q+r)} (-y), 
\end{align}	
where the second equality holds by \eqref{Z_special} and \eqref{Theta_def}.  

\begin{remark}\label{rem_y<0}
(i) Using \eqref{Z2}, for $x<0$, $H^{(q+r)}(x,\theta)=e^{\theta x} > 0$ for $\theta \geq 0$. 
\\
(ii) Using \eqref{fun_ups} together with \eqref{W_a_def}, we have that $\Upsilon(x,y)=W^{(q)}(x-y)$ for $y >0$.
\end{remark}

\begin{remark} 
The function $\Upsilon(x,y)$ will be a key function for the rest of the analysis in this paper. It coincides with $-\Theta^{(q+r)}(x,y)$ when $x < 0$ and with $W^{(q)}(x-y)$ for $y >0$ as in the above remark. To see further relationships with these functions, see \eqref{undershoot_generator} and Lemma \ref{lemma_Upsilon_integral} in the appendix.


\end{remark}




The proof of the following theorem is given in Appendix  \ref{proof_resolvents}.
\begin{theorem} \label{theorem_resolvents}
For $x, b \in \R$, and a positive bounded measurable function $h$ on $\R$ with  compact support 
\begin{align}
\begin{split}\label{g_b_inf_inf}
&\E_x\left[ \int_0^{\infty} e^{-qt} h (U_r^b(t)) \diff t \right]=\int_{-\infty}^\infty  h(y)r^{(q,r)}_b(x,y)\diff y,
\end{split}
\end{align}
where, for $x,y\in\R$, 
\begin{align*}
r^{(q,r)}_b(x,y):= \frac  {q+r} {qr} \frac  {\Phi(q) (\Phi(q+r)- \Phi(q))} {\Phi(q+r)}   Z^{(q,r)} (x-b) 
H^{(q+r)}(b-y, \Phi(q)) 
- \Upsilon(x-b,y-b)
. 
\end{align*}


\end{theorem}
Now using \eqref{controlling_cost} and Theorem \ref{theorem_resolvents}, as well as Lemma \ref{A3_fin} (given in the appendix), we obtain the expression for \eqref{v_b}.
\begin{proposition} \label{prop_v_b}For $x, b \in \R$, the function $v_b(x)$ is finite and can be written
\begin{align}\label{true_v_b}
\begin{split}
v_b(x) 
&=     F(b) Z^{(q,r)}(x-b) - \int_{-\infty}^{\infty} f(y)  \Upsilon(x-b,y-b)
 \diff y - \frac{C r}{q+r} \Big\{ \overline{Z}^{(q)} (x-b) + \frac {\kappa'(0+)} q \Big\} 
 \end{split}
\end{align}
where 
\begin{align}\label{C(b)}
F(b) &:= 
\frac  {\Phi(q+r)- \Phi(q)} {\Phi(q+r)}  \Big[ \frac  {q+r} {qr} \Phi(q) \int_{-\infty}^{\infty} f(y) 
H^{(q+r)}(b-y, \Phi(q)) 
\diff y + \ \frac{C}{\Phi(q)} \Big], \end{align}
which is well-defined and finite by Lemma \ref{A3_fin} and Remark \ref{rem_y<0}(i). 
In particular, 
for $x < b$, from \eqref{def_Upsilon}, 
\begin{multline}\label{v_f_x<b}
v_b(x) 
=  F(b)\frac  { r + q  e^{\Phi(q+r)(x-b)}}{q+r} + \int_{-\infty}^b f(y)  
\Theta^{(q+r)} (x-b, y-b)
 \diff y - \frac{C r}{q+r} \Big\{ x-b + \frac {\kappa'(0+)} q \Big\}.
\end{multline}
\end{proposition}
\begin{proof} 
By  Theorem \ref{theorem_resolvents} and  dominated convergence (due to Lemmas \ref{A3_fin} and \ref{A1_new version}),
 identity \eqref{g_b_inf_inf} holds for $h = f$. 
By this and \eqref{controlling_cost}, the result holds after simplification.  
\end{proof}
\subsection{Polynomial growth of $v_b$.} 
	We conclude this section with the following property of $v_b$, which is required in the verification lemma (Lemma \ref{verificationlemma}).

\begin{lemma} \label{lemma_v_b_polynomial}For each $b \in \R$, $x \mapsto v_b(x)$ is of polynomial growth. 
\end{lemma}
\begin{proof}
 Under $\p$ where $X(0) = 0$ and $z \in \R$, let $U^{b, z}_r$ be the Parisian reflected process with barrier $b \in \R$ driven by $(X(t) + z; t \geq 0)$ and define $R^{b,z}_r$ similarly so that 
$U^{b,z}_r(t) = z + X(t) + R_r^{b,z} (t)$, $t \geq 0$.
Then,  
 \begin{align}
  U^{b,y}_r(t)- U^{b,x}_r(t)
  = (y-x) + (R_r^{b,y} (t) - R_r^{b,x} (t)), \quad x < y. \label{U_diff}
 \end{align}
 
We first show that, for $y > x$,
\begin{align}
U^{b,y}_r(t) - U^{b,x}_r(t) \geq 0, \quad t \geq 0. \label{U_monotone_in_starting_value}
\end{align}
   Let $\sigma:=\inf\{ t > 0: U^{b,x}_r(t) > U^{b,y}_r(t)\}$, and  assume (to derive a contradiction) that $\sigma<\infty$. Because the increments of $U^{b,x}_r$ and $U^{b,y}_r$ can differ only at the jump times of  $R_r^{b,x}$ and $R_r^{b,y}$, we must have that $\Delta R_r^{b,x}(\sigma)>0$ and $U^{b,x}_r(\sigma-)<b$. If $U^{b,y}_r(\sigma-) \leq b$ then $U^{b,x}_r(\sigma) = U^{b,y}_r(\sigma)  = b$.  If 
   	$U^{b,y}_r(\sigma-) >b$ then $U^{b,x}_r(\sigma) = b < U^{b,y}_r(\sigma-) =  U^{b,y}_r(\sigma)$. In both cases, $U^{b,x}_r(\sigma) \leq U^{b,y}_r(\sigma)$ and the inequality holds until the next Poisson arrival time after $\sigma$, which contradicts with the definition of $\sigma$. 
	Hence, we must have $\sigma = \infty$ or equivalently \eqref{U_monotone_in_starting_value}.
   
   On the other hand, letting $\sigma_0 := \inf \{ t > 0: U^{b,x}_r(t) = U^{b,y}_r(t) \}$, we have, for $i \geq 1$ with $T(i) \leq \sigma_0$,  
   \begin{align*}
   \Delta R_r^{b,x} (T(i)) = (b - U_r^{b,x}(T(i)-))^+ \geq (b - U_r^{b,y}(T(i)-))^+ =\Delta R_r^{b,y} (T(i))
   \end{align*}
   while, for $t \geq \sigma_0$, we must have $\Delta R_r^{b,x} (t) = \Delta R_r^{b,y} (t)$. This together with \eqref{U_monotone_in_starting_value} implies that
   	   \begin{align}
   	   	0\leq R_r^{b,x} (t) - R_r^{b,y} (t) \leq y - x, \quad t \geq 0. \label{R_monotone_in_starting_value}
   	   \end{align}
%
%
%
%
By \eqref{U_diff} and \eqref{R_monotone_in_starting_value}, we also have
 \begin{align} \label{U_bound}
 0 \leq U^{b,y}_r(t) - U^{b,x}_r(t)  \leq y-x, \quad t \geq 0.
 \end{align}
 By these bounds and Assumption \ref{assump_f}(i), we have that $v_b$ is of polynomial growth.
\end{proof}

 \section{Selection of $b^*$} \label{section_selection_barrier}
In this section, motivated by the discussion given in Remark \ref{remark_verificatin_lemma}(2), we pursue our candidate barrier $b^*$ such that $v'_{b^*}(b^*) = -C$, and show its existence. 
The convexity of $v_{b^*}$ is shown later in the paper.
\par 

We first obtain the following two lemmas, whose proofs are deferred to Appendices \ref{proof_common_derivative} and \ref{proof_A2}.


\begin{lemma} \label{lemma_Psi_Upsilon}
Define, for $x, y \in \R$,
\begin{align}\label{fun_psi}
\begin{split}
\Psi(x,y) &:= 
W_y^{(q,r)}(x)
- \frac{\Phi(q+r)}{q+r} Z^{(q)}(x, \Phi(q+r)) Z^{(q+r)} (-y).
\end{split}
\end{align}
Then, for $y < b$, 
\begin{align} \label{common_derivative}
\begin{split}
\frac {\partial} {\partial z} \Upsilon(z,y-b) \Big|_{z = (x-b)+} &= - \frac {\partial} {\partial z}  \Psi(x-b,z)\Big|_{z = (y-b)-} \\ &= W^{(q+r)\prime}((x-y)+) - r \int_b^x W^{(q) }(x-z)W^{(q+r)\prime}(z-y)  \diff z \\
&- \Phi(q+r) W^{(q+r)}(b-y)   Z^{(q)}(x-b, \Phi(q+r)).
\end{split}
\end{align}
\end{lemma}

\begin{remark}\label{lemma_limit_func_f}
 By Lemma \ref{A3_fin} and Proposition \ref{prop_v_b}, we must have $\lim_{y \downarrow -\infty} f(y)H^{(q+r)} (b-y,\theta)  = 0$ for $\theta \geq 0$ and $\lim_{y \downarrow -\infty} f(y) \Upsilon(x-b,y-b)  = 0$.
  In addition because 
 \begin{align} \label{Psi_Gamma_diff}
 \Psi(x-b,y-b)=\Upsilon(x-b,y-b)-\frac{\Phi(q+r)}{q+r}Z^{(q)}(x-b,\Phi(q+r)) H^{(q+r)}(b-y),
 \end{align}
  we also have $\lim_{y \downarrow -\infty} f(y)  \Psi(x-b,y-b) = 0$.
\end{remark}
 \begin{lemma}\label{A2}
 	Fix $x, b\in\R$. We can choose $-M\leq b \wedge x$ sufficiently small so that 
 	\begin{align*}
 		\frac{\partial}{\partial x}\int_{-\infty}^{-M}f(y)&
 		\Upsilon(x-b,y-b)%
 		\diff y
 		=\int_{-\infty}^{-M}f(y)\frac{\partial}{\partial x} \Upsilon(x-b,y-b)
 		\diff y.
 	\end{align*}
 \end{lemma}

 Using Lemmas \ref{lemma_Psi_Upsilon} and \ref{A2}, we obtain the results regarding the first derivative of $v_b$.
\begin{lemma} \label{lemma_derivatives} 
Fix $b, x \in \R$. 
(i) We have 
\begin{align}	
 \label{v_b_prime}
	\begin{split} 
		v_b'(x) &=  (q F(b) - f(b)) \frac  {\Phi(q+r)} {q+r}    Z^{(q)}(x-b, \Phi(q+r))    -  \int_b^x W^{(q)}(x-y) f'(y) \diff y\\
		&- \int_{-\infty}^{b} f'(y) 
			\Psi(x-b,y-b)
			\diff y  
		- \frac{Cr}{q+r} Z^{(q)} (x-b).
	\end{split}
\end{align}	

(ii) We have
\begin{align*}
\begin{split}
\E_x\Big[\int_0^{\infty}e^{-qt}&f'(U_r^b(t))\diff t\Big]-v'_b(x) = \left(Z^{(q,r)}(x-b)\frac{q+r}{q}\frac{\Phi(q)}{\Phi(q+r)}-Z^{(q)}(x-b,\Phi(q+r))\right)M^{(q,r)}(b), 
\end{split}
\end{align*}
where 
\begin{align}  \label{M_def}
\begin{split}
M^{(q,r)}(b) &:= \frac  { \Phi(q+r)- \Phi(q)} {r}   \int_{-\infty}^\infty f'(y)
H^{(q+r)}(b-y, \Phi(q)) 
\diff y +\frac{q}{q+r}\frac{\Phi(q+r)}{\Phi(q)}C. 
\end{split}
\end{align}
\end{lemma}

\begin{proof}
(i) By integration by parts, for $x \neq b$,
	\begin{align}
		\frac \partial {\partial x} \int_b^\infty f(y)    W^{(q)}(x-y)  \diff y 
		= f(b) W^{(q)}(x-b) + \int_b^x W^{(q)}(x-y) f'(y) \diff y.\label{int_parts_new}
	\end{align}
Differentiating \eqref{true_v_b} and using \eqref{int_parts_new} and Lemma \ref{A2}
 (with which the derivative can be interchanged over the integral), and that $\Upsilon(x-b,y-b)|_{y = x+}- \Upsilon(x-b,y-b)|_{y=x-} = -W^{(q+r)}(0)$, for $x \neq b$,
\begin{align*}
v_b'(x) &= F(b) Z^{(q,r)\prime}(x-b) - f(b) W^{(q)}(x-b) - \int_b^x W^{(q)}(x-y) f'(y) \diff y  \\
&- \int_{-\infty}^b f(y)  
\frac \partial {\partial x}\Upsilon(x-b,y-b) 
\diff y  -f(x)W^{(q+r)}(0)1_{\{x<b\}} - \frac{C r}{q+r} Z^{(q)} (x-b).
\end{align*}
By  Lemma \ref{lemma_Psi_Upsilon}, Remark \ref{lemma_limit_func_f} and integration by parts and noting that $\Psi(x-b,y-b)|_{y = x+}- \Psi(x-b,y-b)|_{y=x-} = -W^{(q+r)}(0)$,
\begin{multline*}
\int_{-\infty}^b f(y)  \frac \partial {\partial x}\Upsilon(x-b,y-b)
 \diff y = - \int_{-\infty}^b f(y)  \frac \partial {\partial y}\Psi(x-b,y-b) 
 \diff y \\
 = - f(b) \Psi(x-b, 0)+\int_{-\infty}^b f'(y)  \Psi(x-b,y-b) 
 \diff y -f(x)W^{(q+r)}(0)1_{\{x<b\}},
\end{multline*}
where $\Psi(x-b,0) = W^{(q)}(x-b) - \frac{\Phi(q+r)}{q+r}Z^{(q)}(x-b,\Phi(q+r)).$ This together with $Z^{(q,r) \prime}(x-b) =\frac{q}{q+r}  \Phi(q+r) Z^{(q)}(x-b, \Phi(q+r))$ shows \eqref{v_b_prime}. 

For the case $x=b$, following the same computation for the right- and left-hand derivatives, 
it can be confirmed that they both match with \eqref{v_b_prime}.

(ii) 
Integration by parts gives 
\begin{align} \label{int_by_parts_f}
\int_b^{\infty}f(y)e^{-\Phi(q)(y-b)}\diff y= \Big( f(b)+\int_b^{\infty}f'(y)e^{-\Phi(q)(y-b)} \diff y \Big) /\Phi(q),
\end{align}
and by noticing that $\overline{H}^{(q+r)}(z) := \big(H^{(q+r)}(z, \Phi(q))  -\frac{r}{q+r} \frac{\Phi(q+r)}{\Phi(q+r)-\Phi(q)} H^{(q+r)}(z)\big) / {\Phi(q)},$ $z \in \R$, is an antiderivative of $H^{(q+r)} (\cdot, \Phi(q))$ and by Remark \ref{lemma_limit_func_f},
\begin{align*}
\int_{-\infty}^bf(y)
H^{(q+r)}(b-y, \Phi(q)) 
\diff y
&=-\frac{f(b)}{\Phi(q)}\left(1-\frac{r}{q+r}\frac{\Phi(q+r)}{\Phi(q+r)-\Phi(q)}\right) \\ &+ \int_{-\infty}^bf'(y)\overline{H}^{(q+r)}(b-y) \diff y.
\end{align*}
Therefore using the previous identities in \eqref{C(b)} together with Remark \ref{rem_y<0}(i), we obtain
\begin{align}\label{new_c}
\begin{split}
F(b)&=\frac{1}{q} \Big( f(b) - \int_{-\infty}^bf'(y)
H^{(q+r)}(b-y) \diff y \Big) \\
&+\frac{\Phi(q+r)-\Phi(q)}{\Phi(q+r)}\left(\frac{q+r}{qr}\int_{-\infty}^{\infty}f'(y)
H^{(q+r)}(b-y, \Phi(q)) 
\diff y +\frac{C}{\Phi(q)}\right).
\end{split}
\end{align}

Now using \eqref{v_b_prime} and Theorem \ref{theorem_resolvents}  together with  \eqref{Psi_Gamma_diff} and Remark \ref{rem_y<0}(ii), we obtain that 
\begin{align*}
\E_x&\left[\int_0^{\infty}e^{-qt}f'(U_r^b(t))\diff t\right]-v'_b(x) \notag\\
&= 
\frac{Cr}{q+r} Z^{(q)} (x-b) + 
  \frac  {q+r} {qr} \frac  {\Phi(q) (\Phi(q+r)- \Phi(q))} {\Phi(q+r)}  \int_{-\infty}^{\infty}  f'(y)  Z^{(q,r)} (x-b) 
H^{(q+r)}(b-y, \Phi(q)) 
 \diff y \\
 &+\frac{\Phi(q+r)}{q+r}Z^{(q)}(x-b,\Phi(q+r))  \Big( f(b) - \int_{-\infty}^b f'(y)
H^{(q+r)}(b-y) 
 \diff y - q F(b) \Big),
\end{align*}
which shows (ii)
 by \eqref{new_c}.
\end{proof}

From \eqref{Psi_Gamma_diff}, $\Psi(0,y-b)=-\frac{\Phi(q+r)}{q+r}H^{(q+r)}(b-y)$. Hence using \eqref{v_b_prime} and \eqref{new_c},  for any $b\in\R$,
\begin{align} \label{v_b_prime_at_b}
\begin{split}
		v_b'(b) &= - \frac{Cr}{q+r} +  \frac  {\Phi(q+r)} {q+r} \Big(-\int_{-\infty}^bf'(y)
		H^{(q+r)}(b-y) \diff y+q\frac{\Phi(q+r)-\Phi(q)}{\Phi(q+r)}\Big(\frac{C}{\Phi(q)}\\&+\frac{q+r}{qr}\int_{-\infty}^{\infty}f'(y)
		H^{(q+r)}(b-y, \Phi(q)) \diff y \Big)\Big) 
		+ \frac{\Phi(q+r)}{q+r}\int_{-\infty}^{b} f'(y) H^{(q+r)}(b-y)\diff y
		\\
		&=M^{(q,r)}(b)-C.
		\end{split}
\end{align}	
In view of this and Remark \ref{remark_verificatin_lemma}(2), our natural selection of the candidate barrier $b^*$ is such that $M^{(q,r)}(b^*) = 0$. With this choice, the following is immediate by Lemma  \ref{lemma_derivatives}(ii).
\begin{lemma} \label{v_prime_matches} 
	If $b^* \in \R$ is such that $M^{(q,r)}(b^*) = 0$, then 
$v'_{b^*}(x) = \E_x [\int_0^{\infty}e^{-qt}f'(U_r^{b^*}(t))\diff t ]$ for  $x \in \R$. 
\end{lemma}



\subsection{Existence of the optimal barrier $b^*$.}\label{opt_sec}

We first show the following two lemmas. The proof of the first lemma is deferred to Appendix \ref{proof_A3_der}.
\begin{lemma}\label{A3_der}
Fix $b \in\R$ and $\theta \geq 0$.  We can choose $-M< b$ sufficiently small so that
	\begin{align*}
		\frac{\partial}{\partial b}\int_{-\infty}^{-M}& f(y)
		H^{(q+r)}(b-y, \theta)
		\diff y =\int_{-\infty}^{-M} f(y) \frac{\partial}{\partial b}
		H^{(q+r)}(b-y, \theta)
		\diff y.
	\end{align*}
\end{lemma}
%
\begin{lemma} \label{lemma_M_monotone}
For all $b \in \R$ at which $f'(b)$ exists,
\begin{align*}
(e^{-\Phi(q)b}M^{(q,r)}(b))' 
&=- e^{-\Phi(q)b} \frac{\Phi(q+r)}{q+r} \left(Cq + \E\left[f'(\underline{X}(\mathbf{e}_{q+r})+b) \right]\right).
\end{align*}
\end{lemma}
\begin{proof}
Using Lemma \ref{A3_der} together with Remark \ref{rem_y<0}(i) and \eqref{density_running_min},
\begin{align*}
M^{(q,r) \prime}(b)&=
		\frac  {\Phi(q+r)- \Phi(q)} {r}
		\Bigg[-f'(b)+\Phi(q)
		\int_b^{\infty}f'(y)e^{-\Phi(q)(y-b)}\diff y+f'(b)\left(1-\frac{r W^{(q+r)}(0)}{\Phi(q+r)-\Phi(q)}\right)\\
		&+\int_{-\infty}^bf'(y)\left(\Phi(q)Z^{(q+r)}(b-y,\Phi(q))+rW^{(q+r)}(b-y)-\frac{r W^{(q+r)\prime}(b-y)}{\Phi(q+r)-\Phi(q)}\right)\diff y\Bigg]\\
&=\Phi(q)M^{(q,r)}(b)-\frac{\Phi(q+r) }{q+r} \Big( q C+ \E\left[f'(\underline{X}(\mathbf{e}_{q+r})+b) \right] \Big).
\end{align*}
By this, the desired result is immediate.
\end{proof}


\begin{proposition} 
There exists a unique $b^*$ such that $M^{(q,r)}(b^*) = 0$.
\end{proposition}
\begin{proof}
	(i) First we note
	\begin{align} \label{H_f_prime_Phi}
		e^{-\Phi(q)b} \int_{-\infty}^b |f'(y)|
		H^{(q+r)}(b-y, \Phi(q)) 
		\diff y = \int_{-\infty}^0 e^{-\Phi(q)b}  |f'(y+b)|
		H^{(q+r)}(-y, \Phi(q)) 
		\diff y.
	\end{align}
	Because $f'$ is nondecreasing and also of polynomial growth,  $f'((y+b)+)^+ \leq \sum_{0 \leq m \leq N} C_m |y|^m b^{N-m} + K$, $y \in \R$, for some $N \in \mathbb{N}$ and $C_m, K > 0$; similar bounds can be obtained for $f'((y+b)+)^-$. Because $b^k e^{-\Phi(q)b}$ is bounded in $b > 0$ for each $k \geq 0$, we see that $e^{-\Phi(q)b}  |f'((y+b)+)|$ is bounded by a polynomial of $y$ (independent of $b$).  This together with Lemma \ref{A3_fin} allows us to apply dominated convergence, and hence \eqref{H_f_prime_Phi} vanishes as $b \rightarrow \infty$.
	Therefore, in view of \eqref{M_def}, we obtain that
	\begin{equation}\label{lim_b}
	\lim_{b\to\infty}e^{-\Phi(q)b}M^{(q,r)}(b)=0.
	\end{equation}
	
(ii)  By Lemma \ref{lemma_M_monotone} and 
Assumption \ref{assump_f}(i), $
	b \mapsto  l(b) := e^{\Phi(q)b} (e^{-\Phi(q)b}M^{(q,r)}(b))'$ is nonincreasing.   In addition, monotone convergence and Assumption \ref{assump_f}(ii) give 
\begin{align}
\lim_{b \downarrow -\infty}  l(b)
=- \frac{\Phi(q+r)}{q+r} \Big[ Cq + f'(-\infty) \Big] > 0, \quad 
\lim_{b \uparrow \infty}  l(b)
=- \frac{\Phi(q+r)}{q+r} \Big[ Cq + f'(\infty) \Big] < 0.\label{M_limit_infty}
\end{align}
By the positivity of  $\exp (\Phi(q)b)$, there exists $\overline{b} \in \R$ such that $(e^{-\Phi(q)b}M^{(q,r)}(b))' \geq 0$ a.e.\ on $(-\infty, \overline{b})$ and $(e^{-\Phi(q)b}M^{(q,r)}(b))' \leq 0$ a.e.\ on $(\overline{b}, \infty)$; equivalently $b \mapsto e^{-\Phi(q)b}M^{(q,r)}(b)$ is nondecreasing (resp.\ nonincreasing) on $(-\infty, \overline{b})$ (resp.\ $(\overline{b}, \infty)$).
By this and \eqref{lim_b}, there exists $-\infty <  b^* \leq \overline{b}$ such that $e^{-\Phi(q)b}M^{(q,r)}(b)$ (and hence $M^{(q,r)}(b)$ as well) is non-positive on $(-\infty, b^*)$ and non-negative on $(b^*, \infty)$. By the continuity of $M^{(q,r)}(b)$, we must have $M^{(q,r)}(b^*)=0$.

(iii) To conclude, we show the uniqueness of $b^*$.  Because $b^* \leq \overline{b}$, (by the definition of $\overline{b}$) we must have \break $(e^{-\Phi(q)b}M^{(q,r)}(b))'|_{b=b^*+} \geq 0$. Hence it suffices to show that $(e^{-\Phi(q)b}M^{(q,r)}(b))'|_{b=b^*+} \neq 0$ (equivalently $l(b^*+) \neq 0$). Suppose $l(b^*+) = 0$. Then, because $l$ is nonincreasing on $(b^*, \infty)$, $l(b) \leq 0$ a.e.\ on $(b^*, \infty)$ and hence  $e^{-\Phi(q)b}M^{(q,r)}(b) \leq 0$ for $b \in  [b^*, \infty)$.  Because this is also nonnegative by how $b^*$ was chosen, $e^{-\Phi(q)b}M^{(q,r)}(b) = 0$ uniformly on $[b^*, \infty)$, implying $(e^{-\Phi(q)b}M^{(q,r)}(b))' = 0$ a.e.\ on $(b^*, \infty)$, or equivalently, by Lemma \ref{lemma_M_monotone},
$Cq + \E\left[f'(\underline{X}(\mathbf{e}_{q+r})+b)\right] = 0$ for a.e.\ $(b^*, \infty)$, which contradicts with \eqref{M_limit_infty}. 
\end{proof}

\begin{remark}
Using identity \eqref{int_by_parts_f} in \eqref{M_def} leads to 
\begin{multline*}
\frac{q+r}{\Phi(q+r)}M^{(q,r)}(b) =\frac{Cq}{\Phi(q)}+ \frac  {q+r} {r} \frac  {(\Phi(q+r)- \Phi(q))} {\Phi(q+r)} \\ \times \Big[- f(b)+\Phi(q)\int_b^{\infty}f(y)e^{-\Phi(q)(y-b)}\diff y+\int_{-\infty}^bf'(y)
H^{(q+r)}(b-y, \Phi(q))
\diff y \Big].
\end{multline*}
Because monotone convergence and the expression \eqref{H_simple} give
$\lim_{r\to\infty}\int_{-\infty}^b |f'(y)|
H^{(q+r)}(b-y, \Phi(q))
\diff y =0$
and $\lim_{r \rightarrow \infty}\Phi(q+r) = \infty$, we have
\[
\lim_{r\to\infty}\frac{q+r}{\Phi(q+r)}M^{(q,r)}(b)= \tilde{M}^{(q)}(b) := \Phi(q)\int_b^{\infty}f(y)e^{-\Phi(q)(y-b)}\diff y+\frac{Cq}{\Phi(q)}-f(b).
\]
This is consistent with \cite{Y} where the optimal barrier for the classical case is  the root of $\tilde{M}^{(q)}(b) = 0$.
\end{remark}



\section{Proof of optimality} \label{section_optimality}


With $b^* \in \R$ selected in the previous section,  we will prove that our candidate value function $v_{b^*}$ satisfies 
the conditions required in Lemma \ref{verificationlemma} and hence that the strategy $\pi^{b^*}$ is optimal.


\par 
We first confirm the desired smoothness for 
 $v_{b^*}$; we defer the proof to Appendix \ref{App_B}.
\begin{lemma}
	\label{smoothness}
	The function $v_{b^*}$ is sufficiently smooth on $\R$. 
\end{lemma}	
Now in order to verify the equality \eqref{HJB-inequality}, we prove the following.
\begin{lemma} \label{lemma_convexity_slope_at_b}
	The function $v_{b^*}$ is convex, and $v_{b^*}'(b^*) = -C$. 
\end{lemma}
\begin{proof}
 (i) By Assumption \ref{assump_f}(i), 
$f'$ is increasing Lebesgue-a.e.
Hence, using Lemma \ref{v_prime_matches}, together with the monotonicity of $U_r^{b^*}$ in the starting point as in \eqref{U_monotone_in_starting_value}, we obtain 
\begin{align*}
v_{b^*}'(x)=\E_x \Big[\int_0^{\infty}e^{-qt}f'(U_r^{b^*}(t))\diff t \Big]\leq \E_y \Big[\int_0^{\infty}e^{-qt}f'(U_r^{b^*}(t))\diff t \Big]=v_{b^*}'(y)\qquad \text{ for $x<y$.}
\end{align*}
Therefore $v_{b^*}$ is convex.
 
(ii) By how $b^*$ is chosen so that $M^{(q,r)}(b^*) = 0$ and  \eqref{v_b_prime_at_b}, $v_{b^*}'(b^*) = -C$.
\end{proof}

Next, by an application of Lemma \ref{lemma_convexity_slope_at_b}, the following result is immediate.
\begin{proposition}\label{lemma_inf} 
For $x\in\R$, we have 
\begin{equation}\label{infimum}
\mathcal{M}v_{b^*}(x)- v_{b^*}(x)
=
\begin{cases} C(b^*-x)+v_{b^*}(b^*)-v_{b^*}(x) &\mbox{if } x \in (-\infty,b^*), \\ 
0 & \mbox{if } x \in [b^*,\infty). \end{cases}
\end{equation}
\end{proposition}



%

Now we show the following auxiliary result. 
\begin{proposition}\label{lemma_generator}  
	(i) For $x  < b^*$, we have 
		\begin{align*}
	(\mathcal{L}-q)v_{b^*}(x)+f(x)&=- \frac  {qr} {q+r} \Big( F(b^*) \left(1-e^{\Phi(q+r)(x-b^*)}\right) + C(b^*-x) \Big) \notag\\&\displaystyle+r\int_{-\infty}^{b^*}f(y)
	\Theta^{(q+r)}(x-b^*,y-b^*)
	\diff y.
	\end{align*}
	(ii) For $x \geq b^*$, we have $(\mathcal{L}-q)v_{b^*}(x) + f(x) = 0$. 
\end{proposition}
\begin{proof}
(i) Suppose $x < b^*$.
Direct computation gives
$(\mathcal{L}-(q+r))e^{\Phi(q+r)(x-b^*)}=0$, and hence 
\begin{align*} 
(\mathcal{L}-q)
\left(r+q e^{\Phi(q+r)(x-b^*)}\right) 
=
q r(e^{\Phi(q+r)(x-b^*)}-1).
\end{align*}

Let us define, for fixed $z \leq  b^*$, 
\begin{align}
G^{(q+r)}(z):=\E_{z}\Big[\int_0^{\tau_{b^*}^+}e^{-(q+r)t}f(X(t))\diff t\Big] =\int_{-\infty}^{b^*} f(y)
\Theta^{(q+r)}(z-b^*,y-b^*)\diff y, \label{about_G_q_r}
\end{align}
where the last equality holds by \eqref{killed_resolvent}, and is well-defined and finite for all $z \leq b^*$ 
by Remark \ref{remark_finiteness_resolvents}.
With $T_{(-N,b ^*)}:=\inf\{t>0:X(t)\not\in[-N,b ^*]\}$ for $-N<x$, define the processes 
\begin{align*} 
I(t)&:=e^{-(q+r)(t\wedge T_{(-N,b^*)})}G^{(q+r)}(X(t\wedge T_{(-N,b^*)}))+\int_0^{t\wedge T_{(-N,b^*)}}e^{-(q+r)s}f(X(s))\diff s,\quad\text{$t\geq0$}, \\
I(\infty)&:= \lim_{t \rightarrow \infty} I(t) = e^{-(q+r)T_{(-N,b^*)}}G^{(q+r)}(X(T_{(-N,b^*)}))+\int_0^{T_{(-N,b^*)}}e^{-(q+r)s}f(X(s))\diff s.
\end{align*}
Note by the strong Markov property that $G^{(q+r)}(x) = \E_x [I(\infty)]$.



With $(\mathcal{G}(t); t \geq 0)$ being the natural filtration of $X$, we define $\mathbb{P}_x$-martingale: 
$\tilde{I}(t) := \E_x\left[I(\infty)|\mathcal{G}(t)\right]$, $t\geq0$. 
For $x<b^*$ and $t>0$, by the strong Markov property of $X$ and because, on  $\{ t \geq T_{(-N, b^*)} \}$,  $I(t) = I(\infty) = \tilde{I} (t)$, we can write
\begin{align*}
	\tilde{I}(t)  &=1_{\{t<T_{(-N,b ^*)}\}}\Big\{e^{-(q+r)t}\E_{X(t)}\left[I(\infty)\right] +\int_0^te^{-(q+r)s}f(X(s))\diff s\Big\}+1_{\{t\geq T_{(-N,b ^*)}\}} I(t).
\end{align*}
On the other hand, because  $\mathbb{P}_x$-a.s.,
\begin{align*}
	1_{\{t<T_{(-N,b ^*)}\}} I (t)&=1_{\{t<T_{(-N,b ^*)}\}}\Big\{e^{-(q+r)t} \E_{X(t)}\left[I(\infty)\right]+\int_0^te^{-(q+r)s}f(X(s))\diff s\Big\},
\end{align*}
we have that $I = \tilde{I}$, meaning it is  a $\p_x$-martingale.

By Lemma \ref{smoothness} together with the expressions \eqref{v_f_x<b} and \eqref{about_G_q_r}, we have that $G^{(q+r)}$ is sufficiently smooth.
Therefore, using this martingale property and It\^o's formula 
we conclude that
$(\mathcal{L}-q-r)G^{(q+r)}(x)=-f(x)$, 
or equivalently, using the last equality of  \eqref{about_G_q_r}, 
\begin{align*}
(\mathcal{L}-q)\int_{-\infty}^{b^*}&f(y)
\Theta^{(q+r)}(x-b^*,y-b^*)
\diff y + f(x) = r\int_{-\infty}^{b^*}f(y)
\Theta^{(q+r)}(x-b^*,y-b^*)
\diff y.
\end{align*}
Finally, direct computation gives
$(\mathcal{L}-q)\left(b^*-x-\frac{\kappa'(0+)}{q}\right)=-q(b^*-x)$.
Hence putting the pieces together, we complete the proof for the case $x < b^*$.

(ii) Fix $x>b^*$. Similarly to $I$ defined above, the process 
\begin{align*}
e^{-q(t\wedge T_{(b ^*,N)})}v_{b^*}(X(t\wedge T_{(b ^*,N)}))+\int_0^{t\wedge T_{(b ^*,N)}}e^{-qs}f(X(s))\diff s,\qquad\text{$t\geq0$},
\end{align*}
where $T_{(b^*, N)}:=\inf\{t>0:X(t)\not\in[b ^*, N]\}$ with $N > x$,  is a $\p_x$-martingale.
Hence using the martingale property and It\^o's formula (which we can use thanks to the fact that $v_{b^*}$ is sufficiently smooth as in Lemma \ref{smoothness}), we conclude that
$(\mathcal{L}-q)v_{b^*}(x)+f(x) = 0$, as desired. 

For the case $x = b^*$, because $v_{b^*}$ is sufficiently smooth, we obtain the result upon taking $x \rightarrow{b^*}$.
\end{proof}
Now we are ready to show the main result of the paper.
\begin{theorem}
The policy $\pi^{b^*}$ is optimal and the value function is given by $v(x)=v_{b^*}(x)$ for all $x\in\R$.
\end{theorem}
\begin{proof} In view of Lemma \ref{smoothness}, it is sufficient to verify \eqref{HJB-inequality}.

(i) Suppose $x < b^*$.
Using Proposition \ref{lemma_inf} and \eqref{v_f_x<b}, we have
\begin{align*}
\begin{split}
\mathcal{M}v_{b^*}(x)- v_{b^*}(x)
=\frac  {q} {q+r}\left(C(b^*-x)+ F(b^*) \left(1-e^{\Phi(q+r)(x-b^*)}\right)\right) -\int_{-\infty}^{b^*}f(y)
\Theta^{(q+r)}(x-b^*,y-b^*)
\diff y.
\end{split}
\end{align*}
Hence using this and Proposition \ref{lemma_generator}(i), we deduce \eqref{HJB-inequality} for $x < b^*$.
(ii) For the case $x \geq b^*$,  using Proposition \ref{lemma_generator}(ii)
and \eqref{infimum}, we have \eqref{HJB-inequality} as well.
\end{proof}

\begin{remark}
A natural extension 
of the considered problem is to allow additional \emph{fixed ordering costs} 
incurred at each time order is made. In this case,
a ``periodic $(s,S)$-policy'' 
is expected to be optimal. This policy replenishes the item up to the inventory level $S$ at each observation time $\mathcal{T}_r$ whenever it is below the level $s$. This is an interesting and challenging problem and we leave it for further work. 
\end{remark}		
	
\subsection{Numerical examples} 
We now confirm numerically the obtained results using the quadratic inventory cost  $f(x) = x^2$.  In this case, a straightforward computation gives $b^* = \Phi(q+r)^{-1}  -   \Phi(q)^{-1} - \kappa'(0+)/(q+r)   - {q  C} / 2$.
We assume that 
$ X(t) = X(0)+ t+ 0.2 B(t) - \sum_{n=1}^{N(t)} Z_n$, for $0\le t <\infty$. 
Here, $B$ is a standard Brownian motion, $N$ is a Poisson process with arrival rate $1$, and  
$Z$ is an i.i.d.\ sequence of phase-type random variables (whose parameters are given in \cite{leung2015analytic}) approximating the Weibull distribution with shape and scale parameters $2$ and $1$, respectively. The corresponding scale function admits a closed form expression as in \cite{Egami_Yamazaki}. We set $q = 0.05$, $r = 0.5, $ and $C = 1$, unless stated otherwise.

In Figure \ref{fig_optimality}, we plot  $x \mapsto v_b(x)$ for $b = b^*$ and for $b \neq b^*$ along with the points $(b, v_b(b))$. It is confirmed that $v_{b^*}$ is indeed convex (as in Lemma \ref{lemma_convexity_slope_at_b}) and  minimizes over $b$ uniformly in $x$.

In Figure \ref{fig_sensitivity}, we show $v_{b^*}$ for various values of the unit replenishment cost/reward $C$ and the rate of Poisson arrivals $r$, along with those in the continuous monitoring case  \cite{Y}. For the former, as $C$ increases, the value function $v_{b^*}$ increases (uniformly in $x$) while  $b^*$ decreases. On the other hand, as $r$ increases, both $v_{b^*}$ and $b^*$ decrease.  As $r \rightarrow \infty$, the convergence to the case  \cite{Y}  is also confirmed.

 \begin{figure}[htbp]
\begin{center}
\begin{minipage}{1.0\textwidth}
\centering
\begin{tabular}{c}
 \includegraphics[scale=0.35]{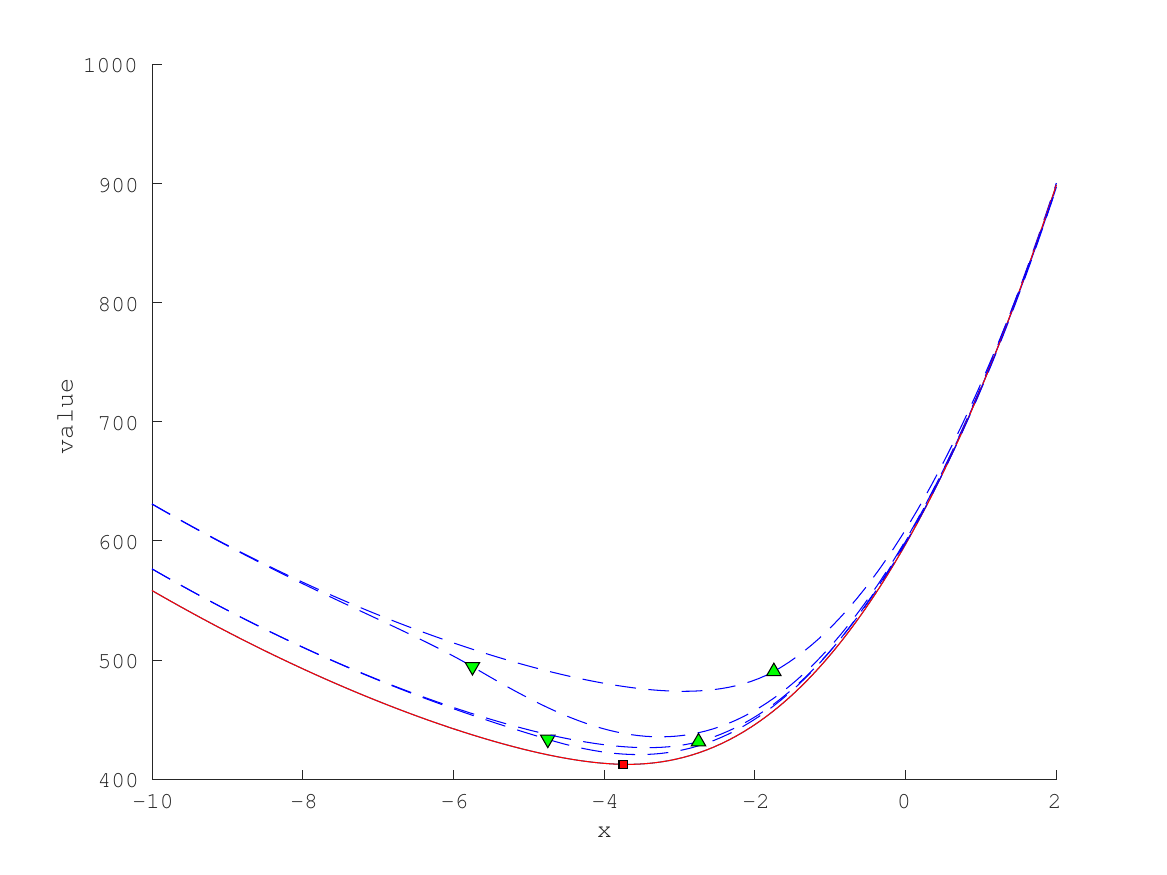}   
\end{tabular}
\end{minipage}
\end{center}
\caption{\footnotesize Plots of $v_{b^*}$ (solid) in comparison to $v_b$ for $b$  $=$  $b^*-2, b^* -1, b^*+1, b^*+2$ (dotted). The point  $(b^*, v_{b^*}(b^*))$ is indicated by a square while the points $(b, v_b(b))$ are indicated by down- and up-pointing triangles for $b < b^*$ and $b > b^*$, respectively.}  \label{fig_optimality}
\end{figure}

 \begin{figure}[htbp]
\begin{center}
\begin{minipage}{1.0\textwidth}
\centering
\begin{tabular}{cc}
 \includegraphics[scale=0.35]{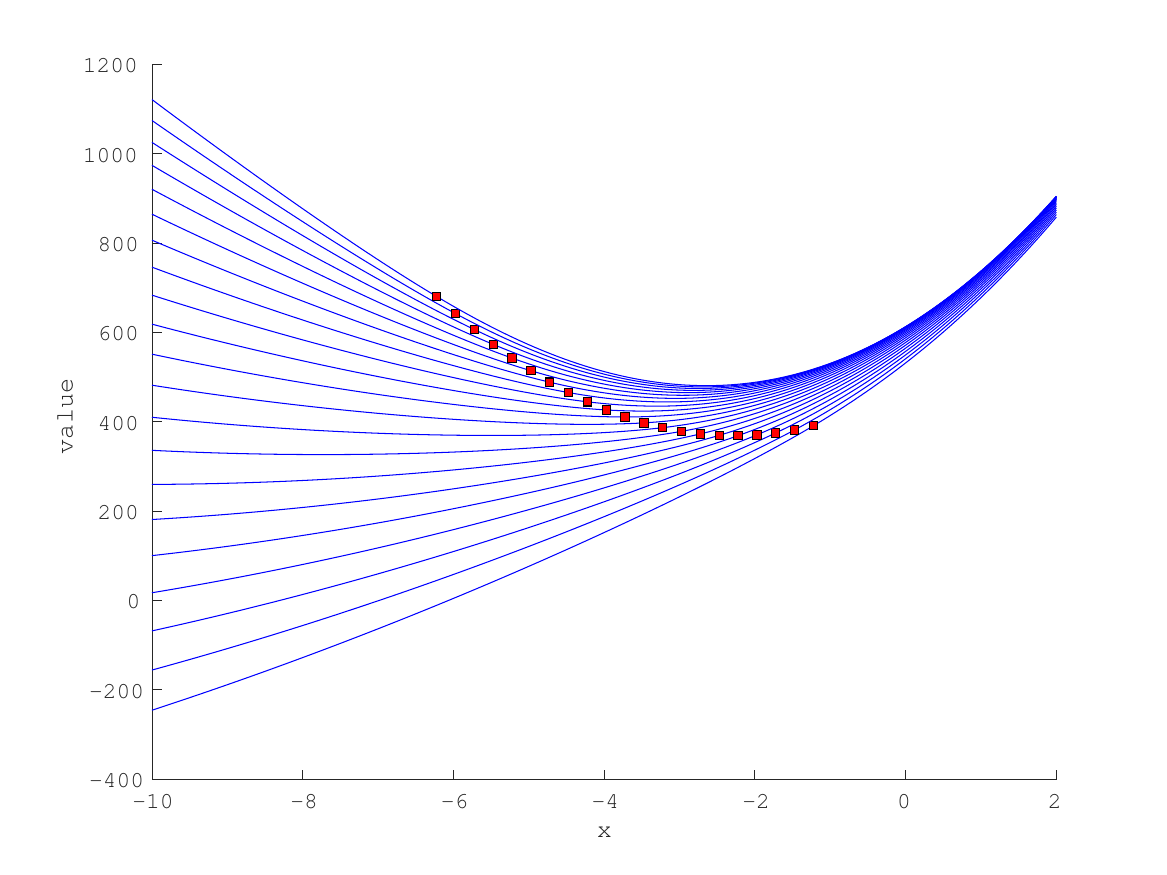}   &  \includegraphics[scale=0.35]{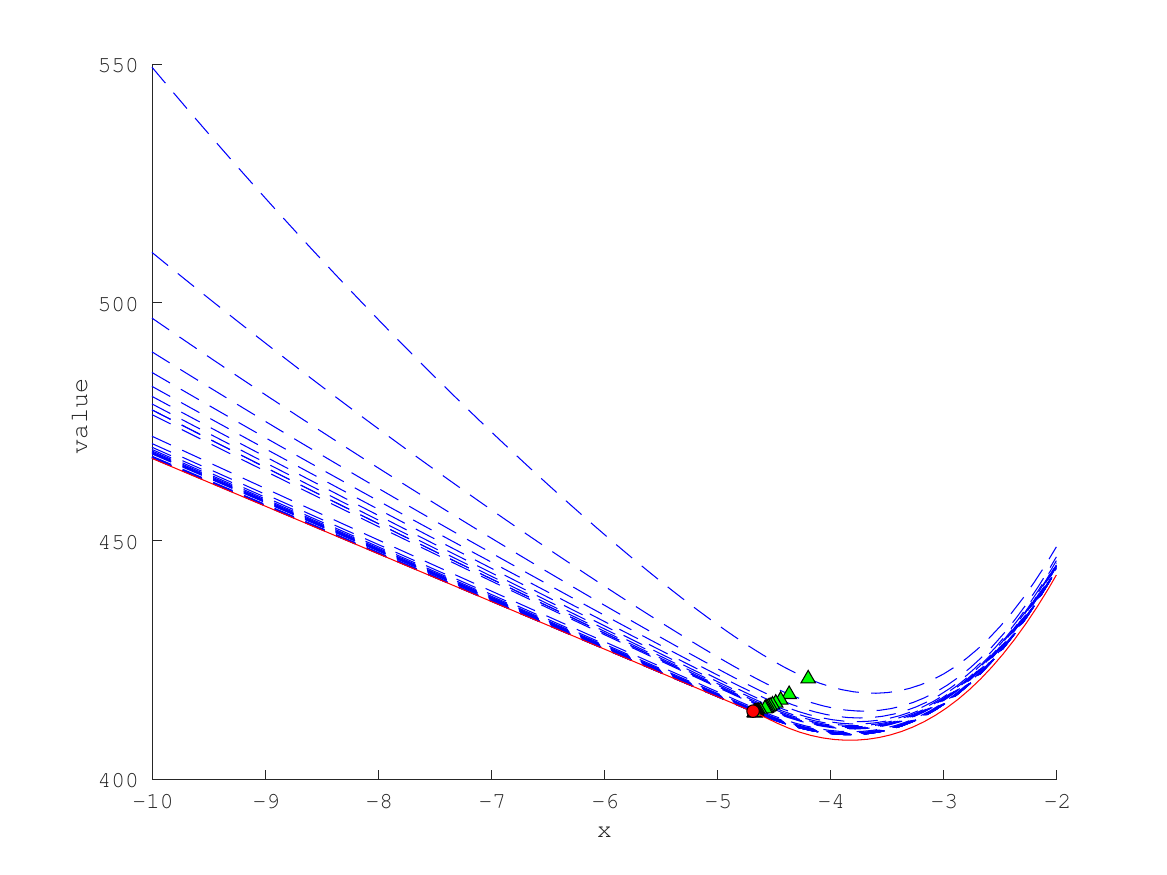}   
\end{tabular}
\end{minipage}
\end{center}
\caption{\footnotesize (Left) Plots of $v_{b^*}$ for  $C=$ $- 100$, $-90$, $\ldots$, $90$, $100$ with $(b^*, v_{b^*}(b^*))$ indicated by squares. (Right) Plots of $v_{b^*}$ (dotted) for  $r = 0.1$, $0.2$, $\ldots$, $0.9$, $1$, $2$, $\ldots$, $9$, $10$, $20$, $\ldots$, $90$, $100$, $200$, $\ldots$, $900$, $1000$ with $(b^*, v_{b^*}(b^*))$ indicated by triangles, along with the continuous monitoring case (solid) with the point at the optimal barrier indicated by a square.}  \label{fig_sensitivity}
\end{figure}

\appendix


\section{Proof of Theorem \ref{theorem_resolvents}}

Recall as in Corollaries 8.7 and 8.8 of \cite{K}, that for any Borel set $A$ on $[0, \infty)$ and on $\R$, respectively, 
\begin{align} \label{resolvent_density}
\E_x \Big[ \int_0^{\tau_{0}^- } e^{-qt} 1_{\left\{ X(t) \in A \right\}} \diff t\Big] &= \int_A \Big[ e^{-\Phi(q) y}W^{(q)}(x) -{W^{(q)}} (x-y) \Big] \diff y, \quad x \geq 0, \\
\label{resolvent_density_2}
	\E_x\left[ \int_0^{\infty}e^{-(q+r)t } 1_{\{X(t) \in A\}}\diff t\right] &=\int_A \left[\frac{e^{\Phi(q+r)(x-y)}}{\kappa'(\Phi(q+r))
	}-W^{(q+r)}(x-y)\right]\diff y, \quad x \in \R.
\end{align} 

\subsection{Proof of Theorem \ref{theorem_resolvents}} \label{proof_resolvents}
For $x\in\R$, let us denote the left-hand side of \eqref{g_b_inf_inf} by $g_b(x)$ and in particular $g(x):=g_0(x)$. We will prove the result for $b=0$; the general case follows because the spatial homogeneity of the L\'evy process implies that 
$g_b(x)=\E_{x-b}\left[ \int_0^{\infty } e^{-qt} h (U_r^0(t)+b) \diff t \right]$.
	
(i) For $x \in \R$,  by the strong Markov property,
\begin{align}\label{MPa}
g(x)&=\E_x\Big[ \int_0^{\tau_0^-  } e^{-qt} h(X(t)) \diff t \Big] + \E_x\left[ e^{-q\tau_0^-}g(X(\tau^-_0))  1_{\{   \tau^-_0<\infty\}} \right].
\end{align}
In particular, for $x < 0$, again by the strong Markov property and because $U_r^0  = X$ on $[0, T(1) \wedge \tau_0^+)$, 
\begin{align*}
g(x) = A (x) g(0)+ B(x), 
\end{align*} 
where, for $x \leq 0$, 
\begin{align*}
A(x) &:= \E_x \Big[ e^{-q (\tau_0^+ \wedge T(1))}  \Big] = \frac{r}{q+r}+\frac{q}{q+r}e^{\Phi(q+r)x},  \nonumber \\
\begin{split}
B(x)&:
= \E_x \Big[ \int_0^{\tau_0^+} e^{-qt} 1_{\{ t < T(1)\}} h(X(t))     \diff t\Big] 
 =\int_{-\infty}^0 h(y) 
 \Theta^{(q+r)}(x,y)
 \diff y.
 \end{split}
\end{align*}
Here, the second equality of the former holds by the fact that $T(1)$ is an independent exponential random variable with parameter $r$ and Theorem 3.12 of \cite{K}. The second equality of the latter holds by \eqref{killed_resolvent}. 

Now applying identity (3.19) in \cite{APP2007},
\begin{align*}
\E_{x}&\left[e^{-q\tau_0^-}A(X(\tau_0^-))1_{\{\tau_0^-<\infty\}}\right]\\
&=\frac{r}{q+r}\left(Z^{(q)}(x)-\frac{q}{\Phi(q)}W^{(q)}(x)\right)+\frac{q}{q+r}\left(Z^{(q)}(x,\Phi(q+r))-\frac{r W^{(q)}(x)}{\Phi(q+r)-\Phi(q)}\right)\\
&=Z^{(q,r)}(x)- \frac {qr} {q+r} \frac {\Phi(q+r)} {\Phi(q) (\Phi(q+r)- \Phi(q))} W^{(q)}(x).
\end{align*}
In addition, 
using identity (5) in \cite{AI} together with Lemma 2.1 in \cite{LRZ},
we obtain for $c > x$,
\begin{align}\label{undershoot_generator}
\E_{x}\left[e^{-q\tau_0^-}B(X(\tau_0^-))1_{\{\tau_0^-<\tau_c^+\}}\right]
&= \int_{-\infty}^0 h(y)   \E_x\left[e^{-q\tau_0^-}
\Theta^{(q+r)}(X(\tau_0^-),y)
1_{\{\tau_0^-<\tau_c^+\}}\right]  \diff y \notag\\
&= -\int_{-\infty}^0 h(y) \Upsilon(x, y)
\diff y +\frac{W^{(q)}(x)}{W^{(q)}(c)}\int_{-\infty}^0 h(y)  \Upsilon(c,y)
\diff y.
\end{align}
By  \eqref{Z_special} and Remark \ref{remark_smoothness_zero}(3), for $y \in \R$, 
	$\lim_{x \rightarrow \infty}W^{(q,r)}_y(x)/W^{(q)}(x)
	=Z^{(q+r)}(-y, \Phi(q))$.
Also following the proof of Corollary 3.2(iii) in \cite{APY} we have
$\lim_{x\to\infty}Z^{(q)}(x,\Phi(q+r))/W^{(q)}(x)=r/(\Phi(q+r)-\Phi(q))$. 
Hence by taking $c\uparrow\infty$ in \eqref{undershoot_generator} and using these limits, we get 
\begin{align*}
\E_{x}\left[e^{-q\tau_0^-}B(X(\tau_0^-))1_{\{\tau_0^-<\infty\}}\right]
=-\int_{-\infty}^0 h(y)  \Upsilon(x,y) \diff y
+W^{(q)}(x)\int_{-\infty}^0 h(y)  H^{(q+r)}(-y, \Phi(q))
\diff y.
\end{align*}
Substituting these in \eqref{MPa} and by \eqref{resolvent_density} and Remark \ref{rem_y<0},
\begin{multline}\label{MPa2}
g(x)=
g(0)\left\{Z^{(q,r)}(x)- \frac {qr} {q+r} \frac {\Phi(q+r)} {\Phi(q) (\Phi(q+r)- \Phi(q))}W^{(q)}(x)\right\}
\\+W^{(q)}(x)\int_{-\infty}^\infty h(y)  H^{(q+r)}(-y, \Phi(q)) \diff y -\int_{-\infty}^\infty h(y)  \Upsilon(x,y)\diff y.
\end{multline}

(ii) On the other hand, by the strong Markov property, we can also write 
\begin{align} \label{g_zero_recursion}
g(0) = \gamma_1 + \gamma_2 g(0) +  \gamma_3,
\end{align}
where 
\begin{align*}
\gamma_1 := \E \Big[\int_0^{T(1)} e^{-qt} h(X(t)) \diff t \Big], \; \gamma_2 := \E \Big[e^{-q T(1)} 1_{\{X(T(1)) \leq 0\}} \Big], \; \gamma_3 :=  \E \Big[e^{-q T(1)}  g(X(T(1)) 1_{\{ X(T(1)) > 0\}}   \Big],
\end{align*}
whose values are to be computed below.

(1) We get
$\gamma_1 =\E \left[\int_0^\infty 1_{\{t < T(1) \}}e^{-qt} h(X(t)) \diff t \right] = \E [\int_0^{\infty} e^{-(q+r) t} h(X(t)) \diff t ]$. 

(2) Using \eqref{resolvent_density_2}, we obtain 
\begin{align*}
\gamma_2 = r \Big( \frac{1}{q+r}- \E \left[\int_0^{\infty}e^{-(q+r)s} 1_{\{X(s) \geq 0\}}\diff s\right] \Big)
	&=r \Big( \frac{1}{q+r}-\frac{1}{\kappa'(\Phi(q+r))}\frac{1}{\Phi(q+r)} \Big).
\end{align*}

(3) Again by \eqref{resolvent_density_2},
\begin{align}
\gamma_3 =r\E \Big[\int_0^{\infty}e^{-(q+r)s}g(X(s)) 1_{\{ X(s) > 0\}}\diff s\Big]=\frac{r}{\kappa'(\Phi(q+r))}\int_0^{\infty}
e^{-\Phi(q+r)y}g(y)\diff y, \label{gamma_3_rewrite}
\end{align}
which we shall compute using the expression of $g$ as in \eqref{MPa2}. First, by 
 integration by parts, 
\begin{align*}
\int_0^{\infty}e^{-\Phi(q+r)y}Z^{(q)}(y)\diff y
&=\frac{1}{\Phi(q+r)}\left(1+q\int_0^{\infty}e^{-\Phi(q+r)u}W^{(q)}(u)\diff u\right) =\frac{1}{\Phi(q+r)}\frac{q+r}{r}.
\end{align*}
Because \eqref{scale_function_laplace} and  \eqref{Z_special} give $e^{-\Phi(q+r)y}Z^{(q)}(y,\Phi(q+r)) = r\int_y^{\infty}e^{-\Phi(q+r)z}W^{(q)}(z)\diff z$,
\begin{align} \label{Z_Phi_integral}
\begin{split}
	\int_0^{\infty}e^{-\Phi(q+r)y}Z^{(q)}(y,\Phi(q+r))\diff y 
	&=r\int_0^{\infty}\int_y^{\infty}e^{-\Phi(q+r)z}W^{(q)}(z)\diff z \diff y\\
	&=r\int_0^{\infty}ze^{-\Phi(q+r)z}W^{(q)}(z)\diff z =\frac{\kappa'(\Phi(q+r))}{r},
	\end{split}
\end{align}
where the second equality holds by the change of variables and the last holds because monotone convergence and \eqref{scale_function_laplace} give that  $\int_0^{\infty}ze^{-\theta z}W^{(q)}(z)\diff z =-\frac{\partial}{\partial\theta} \int_0^{\infty}e^{-\theta z}W^{(q)}(z)\diff z$. 



\begin{lemma} \label{lemma_Upsilon_integral}
For $y \in \R$,
$\int_0^{\infty}e^{-\Phi(q+r)x}\Upsilon(x,y)\diff x= [{e^{-\Phi(q+r)y}}  -W^{(q+r)}(-y) {\kappa'(\Phi(q+r))} ]/ {r}$. 
\end{lemma}
\begin{proof} 


We have for $\theta > \Phi(q+r)$, by the convolution theorem,
\begin{multline*}
\int_0^\infty e^{-\theta x} W_y^{(q,r)}(x)
\diff x = \Big( \int_0^\infty e^{-\theta x}  W^{(q+r)}(x-y) \diff x \Big) \Big( 1 - r \int_0^\infty e^{- \theta x} W^{(q)}(x) \diff x\Big) \\
=\Big( \frac{ e^{-\theta y}} {\kappa(\theta) - q-r} -  \int_{y\wedge0}^0 e^{-\theta x}  W^{(q+r)}(x-y) \diff x \Big) \frac {\kappa(\theta) - q -r} {\kappa(\theta) - q}
\xrightarrow{\theta \downarrow \Phi(q+r)}  \frac{ e^{-\Phi(q+r) y}}  {r}. 
\end{multline*}
This together with  \eqref{Z_Phi_integral} completes the proof.
\end{proof}
By this lemma, Fubini's theorem and \eqref{resolvent_density_2},
\begin{align*}
\int_0^{\infty}e^{-\Phi(q+r)y}\int_{-\infty}^{\infty}h(z)\Upsilon(y,z)\diff z \diff y
=\frac{\kappa'(\Phi(q+r))}{r}\E\left[\int_0^{\infty}e^{-(q+r)t}h(X(t)) \diff t \right].
\end{align*}
Substituting these and with the help of \eqref{MPa2} in \eqref{gamma_3_rewrite}, 
\begin{multline*}
\gamma_3 
=g(0)\frac{r}{\kappa'(\Phi(q+r))}\Big[\frac{1}{\Phi(q+r)}+\frac{q}{q+r}\frac{\kappa'(\Phi(q+r))}{r}-\frac {q} {q+r} \frac {\Phi(q+r)} {\Phi(q) (\Phi(q+r)- \Phi(q))}\Big]\\
-\E\left[\int_0^{\infty}e^{-(q+r)t}h(X(t))\diff t \right]+\frac{1}{\kappa'(\Phi(q+r))} \int_{-\infty}^\infty h(y)  H^{(q+r)}(-y, \Phi(q)) \diff y.
\end{multline*}
Now substituting the computed values of $\gamma_1$, $\gamma_2$, and $\gamma_3$  in \eqref{g_zero_recursion}
 and after simplification, we have
\begin{align*}
g(0)
&=g(0)-\frac {rq} {q+r} \frac {\Phi(q+r)} {\Phi(q) (\Phi(q+r)- \Phi(q))}\frac{g(0)}{\kappa'(\Phi(q+r))}+\frac{1}{\kappa'(\Phi(q+r))}\int_{-\infty}^\infty h(y)  H^{(q+r)}(-y, \Phi(q)) \diff y,
\end{align*}
and hence, solving for $g(0)$ we obtain
\begin{align*}
	g(0)=
	\frac  {q+r} {qr} \frac  {\Phi(q) (\Phi(q+r)- \Phi(q))} {\Phi(q+r)}
	\int_{-\infty}^\infty h(y) H^{(q+r)}(-y, \Phi(q))  
	\diff y.
\end{align*}
Substituting this back in \eqref{MPa2}, we have \eqref{g_b_inf_inf} for $b = 0$, as desired. 

\section{Integrability results}

\begin{lemma}\label{A3_fin}
	Consider $g:\R\to\R$ that satisfies Assumption \ref{assump_f}(i). 
	Then, for any $b\in\R$ and $\theta\geq0$, we have 
	$\int_{-\infty}^{b} | g(y)|H^{(q+r)}(b-y, \theta)
	\diff y<\infty$.
\end{lemma}
\begin{proof}
By identity \eqref{H_theta}, 
\begin{align*}
	&\int_{-\infty}^{b} | g(y) |
	H^{(q+r)}(b-y, \theta)
	\diff y 
	=\int_{-\infty}^{b}| g(y)|\E_{b-y}\left[ e^{-(q+r)\tau_0^-+\theta X(\tau_0^-)}1_{\{\tau_0^-<\infty\}}\right] \diff y \\
	&\leq\int_{-\infty}^{b}| g(y)|\mathbb{P}(-\underline{X}(\mathbf{e}_{q+r})>b-y)\diff y =\int_{-\infty}^{b}| g(y)|\int_{b-y}^{\infty}\mathbb{P}(-\underline{X}(\mathbf{e}_{q+r})\in \diff z)\diff y\\
	&=\int_{0}^{\infty}| g(b-u)|\int_{u}^{\infty}\mathbb{P}(-\underline{X}(\mathbf{e}_{q+r})\in \diff z)\diff u=\int_{0}^{\infty}\mathbb{P}(-\underline{X}(\mathbf{e}_{q+r})\in \diff z)\int_0^{z}| g(b-u)|\diff u.
\end{align*}
Here, as in (3.11) of \cite{HPY} (using Assumption \ref{assump_levy_measure}), we have $\E [e^{-\theta \underline{X}(\mathbf{e}_{q+r})}] < \infty$ for $0 < \theta < \bar{\theta}$. This together with the polynomial growth of $g$ as in Assumption \ref{assump_f}(i), implies that the above is finite.
\end{proof}
%


\begin{lemma}\label{A1_new version}
	Fix any $b\in\R$. 
	(i)   
	For any $x\geq b$, 
	$\sup_{z\in[0,x-b]}\int_{-\infty}^{b}|f(y)| |\Theta^{(q+r)}(z,y-b)|\diff y<\infty$. 
	(ii) For any $x\in\R$, 
	$\int_{-\infty}^{b}|f(y)||\Upsilon(x-b,y-b)|\diff y<\infty$.
\end{lemma}
\begin{proof}
	(i)	
	Recall Remark \ref{remark_theta_sign}.  For $z\in[0,x-b]$, because $ b \leq z +b$, by \eqref{killed_resolvent} and following similar arguments as in \eqref{f_and_theta},
	\begin{multline*} 
	\int_{-\infty}^{b}|f(y)| |\Theta^{(q+r)}(z,y-b) | \diff y \leq \int_{-\infty}^{z+b}|f(y)||\Theta^{(q+r)}(z,y-b)|\diff y\\
	=e^{\Phi(q+r)z}\mathbb{E}_{b}\Big[\int_0^{\tau_{z+b}^+}e^{-(q+r)t}|f(X(t))|\diff t\Big] \leq e^{\Phi(q+r)(x-b)}\mathbb{E}_{b}\Big[\int_0^{\tau_{x}^+}e^{-(q+r)t}|f(X(t))|\diff t\Big],
	\end{multline*}
	and hence we have the result by Remark \ref{remark_finiteness_resolvents}.
	
	(ii) Fix $x<b$. Then by Remark \ref{remark_theta_sign}(i) and \eqref{def_Upsilon}, we have for $y<b$ that $|\Upsilon(x-b,y-b)|=\Theta^{(q+r)}(x-b,y-b)$. Hence
	\begin{align*}
		\int_{-\infty}^{b}|f(y)| |\Upsilon(x-b,y-b) | \diff y=\int_{-\infty}^{b}|f(y)| \Theta^{(q+r)}(x-b,y-b)\diff y=\E_x\Big[\int_0^{\tau_b^+}e^{-(q+r)t}|f(X(t))|\diff t\Big],
	\end{align*}
	which is finite by Remark \ref{remark_finiteness_resolvents}.
	
	 On the other hand, for $x\geq b$, we 
    note that, by an application of Fubini's theorem and (i),
	\begin{align*}
	\int_{-\infty}^{b}|f(y)|&\int_0^{x-b}W^{(q)}(x-b-z)|\Theta^{(q+r)}(z,y-b)|\diff z\diff y\\&=
	\int_0^{x-b}W^{(q)}(x-b-z)\int_{-\infty}^{b}|f(y)||\Theta^{(q+r)}(z,y-b)|\diff y\diff z\\&\leq \overline{W}^{(q)}(x-b)\sup_{z\in[0,x-b]}\int_{-\infty}^{b}|f(y)||\Theta^{(q+r)}(z,y-b)|\diff y<\infty.
	\end{align*}
	In view of the form of $\Upsilon$ as in \eqref{def_Upsilon} and (i), the proof is complete.
	\end{proof}

\section{Other  proofs}\label{sec_A1}

\subsection{Proof of Lemma \ref{lemma_Psi_Upsilon}} \label{proof_common_derivative} 
For $y < b$, because
\begin{align*}
	\frac {\partial} {\partial u} W_{u-b}^{(q,r)}(x-b) \Big|_{u = y-} 
	=-W^{(q+r)\prime}((x-y)+)+ r\int_{b}^{x}  W^{(q)}(x-z)W^{(q+r)\prime}(z-y)\diff z,
\end{align*}
we have that $-\frac {\partial} {\partial z}  \Psi(x-b,z)|_{z = (y-b)-}$ reduces to the right hand side of  \eqref{common_derivative}.

On the other hand, we obtain by integration by parts
\begin{align}\label{aux_1_app}
	\frac {\partial} {\partial z} W_{y-b}^{(q,r)}(z)\Big|_{z = (x-b)+} 
	&=W^{(q+r)\prime}((x-y)+)-rW^{(q)}(x-b)W^{(q+r)}(b-y)\notag\\&- r\int_{b}^{x}  W^{(q)}(x-z)W^{(q+r)\prime}(z-y)\diff z.
\end{align}
Using 
$\frac {\partial} {\partial z}Z^{(q)}(z,\Phi(q+r))=\Phi(q+r)Z^{(q)}(z,\Phi(q+r))-rW^{(q)}(z)$ 
 and \eqref{aux_1_app}  in \eqref{fun_ups}, we have that $\frac {\partial} {\partial z} \Upsilon(z,y-b) |_{z = (x-b)+}$ equals the right hand side of \eqref{common_derivative}.
\subsection{Proof of Lemma \ref{A2}} \label{proof_A2}
%

(i)
Fix $y < b \wedge x$ and $\varepsilon > 0$.
With $W_{\Phi(q+r)}$ defined as in Remark  \ref{remark_smoothness_zero}(3),
\begin{multline}\label{der-int_3}
\frac{\Theta^{(q +r)}(x-b+\varepsilon,y-b)-\Theta^{(q+r)}(x-b,y-b)}{\varepsilon}
=\frac{e^{\Phi(q+r)\varepsilon}-1}{\varepsilon}\Theta^{(q +r)}(x-b,y-b)  \\ -e^{\Phi(q+r)(x+\varepsilon-y)}\left(\frac{W_{\Phi(q+r)}(x+\varepsilon-y)-W_{\Phi(q+r)}(x-y)}{\varepsilon}\right).
\end{multline}
Here we note that $\varepsilon \mapsto(e^{\Phi(q+r)\varepsilon}-1)/\varepsilon$ is bounded in compact sets on $(0, \infty)$,
 and   
 that $\int_{-\infty}^{b}\big|f(y)| | \Theta^{(q+r)}(x-b,y-b)\big|\diff y < \infty$ by Lemma \ref{A1_new version}(i).


As in Appendix A.1 of \cite{HPY} (page 1150) we have that 
\begin{align*}
y \mapsto |f(y)|e^{-\Phi(q+r)y}\Bigg|\frac{W_{\Phi(q+r)}(u+\varepsilon-y)-W_{\Phi(q+r)}(u-y)}{\varepsilon}\Bigg|
\end{align*}
 is bounded in $\varepsilon > 0$ by a function integrable over $(-\infty, -M)$ for some $- M < b \wedge x$.
Therefore dominated convergence gives
\begin{align}\label{der-int_5}
\frac{\partial}{\partial x}\int_{-\infty}^{-M}f(y)\Theta^{(q+r)}(x-b, y-b)\diff y&
=\lim_{\varepsilon\downarrow 0}\int_{-\infty}^{-M}f(y)\frac{\Theta^{(q+r)}(x-b+\varepsilon,y-b)-\Theta^{(q+r)}(x-b,y-b)}{\varepsilon}\diff y\notag\\
&=\int_{-\infty}^{-M}f(y)\frac{\partial}{\partial x}\Theta^{(q+r)}(x-b, y-b)\diff y.
\end{align}

(ii) Fix $x > b$ and consider the second term of $\Upsilon$ in \eqref{def_Upsilon}.
We take $\delta>0$ small enough so that $x - b - \delta > 0$.
By Fubini's theorem
\begin{align*}
\int_{-\infty}^{-M}f(y)&\int_0^{x-b}W^{(q)}(x-b-z) \Theta^{(q+r)}(z,y-b)\diff z\diff y\\&=\int_0^{x-b}W^{(q)}(x-b-z)\int_{-\infty}^{-M}f(y)\Theta^{(q+r)}(z,y-b)\diff y\diff z.
\end{align*}

On the other hand, by the mean value theorem 
and Lemma \ref{A1_new version}(i), for $0 < z < x-b - \delta$ 
and $0 < \varepsilon < \bar{\varepsilon}$,
\begin{multline*}
\Big|\frac{W^{(q)}(x-b-z+\varepsilon)-W^{(q)}(x-b-z)}{\varepsilon}\int_{-\infty}^{-M}f(y)\Theta^{(q+r)}(z,y-b)\diff y\Big|\\ \leq \sup_{u\in[\delta,x-b + \bar{\varepsilon}]}W^{(q)\prime}(u+)\sup_{z\in[0,x-b]} \Big|\int_{-\infty}^{-M}f(y)\Theta^{(q+r)}(z,y-b)\diff y \Big|<\infty.
\end{multline*}
This and  dominated convergence imply
\begin{multline*}
\lim_{\varepsilon\downarrow 0}\int_{0}^{x-b-\delta}\frac{W^{(q)}(x-b-z+\varepsilon)-W^{(q)}(x-b-z)}{\varepsilon}\int_{-\infty}^{-M}f(y)\Theta^{(q+r)}(z,y-b) \diff y \diff z\\
=\int_{0}^{x-b-\delta}W^{(q)\prime}(x-b-z)\int_{-\infty}^{-M}f(y)\Theta^{(q+r)}(z,y-b) \diff y \diff z.
\end{multline*}
On the other hand, 
\begin{align*}
\Big|\int_{x-b-\delta}^{x-b}&\frac{W^{(q)}(x-b+\varepsilon-z)-W^{(q)}(x-b-z)}{\varepsilon}\int_{-\infty}^{-M}f(y)\Theta^{(q+r)}(z,y-b)\diff y \diff z\Big|\\
&\leq \Big( \sup_{z\in[0,x-b]}\int_{-\infty}^{-M} | f(y)| |\Theta^{(q+r)}(z,y-b) |\diff y \Big) \int_{x-b-\delta}^{x-b}\frac{W^{(q)}(x-b+\varepsilon-u)-W^{(q)}(x-b-u)}{\varepsilon}\diff u,
\end{align*}
which vanishes as $\varepsilon \downarrow 0$ and then $\delta \downarrow 0$ because
l'Hopital's rule gives
\begin{align*}
\int_{x-b-\delta}^{x-b}\frac{W^{(q)}(x-b+\varepsilon-z)-W^{(q)}(x-b-z)}{\varepsilon}\diff z &= \frac{\overline{W}^{(q)}(\varepsilon+\delta)-\overline{W}^{(q)}(\delta)-\overline{W}^{(q)}(\varepsilon)}{\varepsilon}  \\ &\xrightarrow{\varepsilon \downarrow 0} W^{(q)}(\delta)-W^{(q)}(0).
\end{align*}
Putting the pieces together we obtain
\begin{align*}
A_1 := \lim_{\varepsilon\downarrow 0}\int_{0}^{x-b}&\frac{W^{(q)}(x-b+\varepsilon-z)-W^{(q)}(x-b-z)}{\varepsilon}\int_{-\infty}^{-M}f(y)\Theta^{(q+r)}(z,y-b)\diff y \diff z \\
&=\lim_{\delta\downarrow 0}\lim_{\varepsilon\downarrow 0}\int_{0}^{x-b-\delta}\frac{W^{(q)}(x-b+\varepsilon-z)-W^{(q)}(x-b-z)}{\varepsilon}\int_{-\infty}^{-M}f(y)\Theta^{(q+r)}(z,y-b)\diff y \diff z \\
&+\lim_{\delta\downarrow 0}\lim_{\varepsilon\downarrow 0}\int_{x-b-\delta}^{x-b}\frac{W^{(q)}(x-b+\varepsilon-z)-W^{(q)}(x-b-z)}{\varepsilon}\int_{-\infty}^{-M}f(y)\Theta^{(q+r)}(z,y-b)\diff y \diff z \\
&=\int_{0}^{x-b}W^{(q)\prime}(x-b-z)\int_{-\infty}^{-M}f(y)\Theta^{(q+r)}(z,y-b)\diff y \diff z.
\end{align*}
On the other hand, by \eqref{der-int_3} the mapping 
$z \mapsto\int_{-\infty}^{-M}f(y) \Theta^{(q+r)}(z,y-b)\diff y$
is continuous, and hence
\begin{multline*}
A_2 :=\lim_{\varepsilon\downarrow 0}\frac{1}{\varepsilon}\int_{x-b}^{x-b+\varepsilon}W^{(q)}(x-b+\varepsilon-z)\int_{-\infty}^{-M}f(y) \Theta^{(q+r)}(z,y-b)\diff y \diff z  \\ =W^{(q)}(0)\int_{-\infty}^{-M}f(y)\Theta^{(q+r)}(x-b,y-b) \diff y \diff z.
\end{multline*}
Therefore 
\begin{multline}\label{der-int_4}
\frac{\partial}{\partial x}  \int_{-\infty}^{-M}f(y) \int_0^{x-b}W^{(q)}(x-b-z) \Theta^{(q+r)}(z,y-b)\diff z \diff y
= A_1 + A_2 \\
=\int_{-\infty}^{-M}f(y)\frac{\partial}{\partial x}\int_0^{x-b}W^{(q)}(x-b-z)\Theta^{(q+r)}(z,y-b) \diff z \diff y. 
\end{multline}
We now conclude the proof by identities \eqref{der-int_5}, \eqref{der-int_4} and \eqref{def_Upsilon}.

\subsection{Proof of Lemma \ref{A3_der}} \label{proof_A3_der} 
We have
\begin{align*}
&
\frac{\partial}{\partial z}
H^{(q+r)}(z, \theta) \big|_{z=(b-y)+} 
=\theta  H^{(q+r)}(b-y, \theta)
-\frac{\kappa(\theta)-(q+r)}{\theta-\Phi(q+r)}
\underline{r}^{(q+r)}((b-y)+),
\end{align*}	
where $\underline{r}^{(q+r)}(x) := W^{(q+r)\prime}(x)-\Phi(q+r)W^{(q+r)}(x) > 0$, $x > 0$, with $ (q+r)\underline{r}^{(q+r)}(x) / \Phi(q+r)$ being the density function 
of  $-\underline{X}(\mathbf{e}_{q+r})$ 
as in \eqref{density_running_min},
and hence
\begin{align*}
\Big|
\frac{\partial}{\partial z} H^{(q+r)}(z, \theta) \big|_{z=(b-y)+}
\Big|
&\leq\theta
H^{(q+r)}(b-y, \theta)
+\left|\frac{\kappa(\theta)-(q+r)}{\theta-\Phi(q+r)}\right|
\underline{r}^{(q+r)}((b-y)+).
\end{align*}

Let us suppose that $b\in[b_1,b_2]$ with $b_1 >-M$. First, by \eqref{H_theta}, we have that
 $H^{(q+r)}(b-y, \theta)
\leq \E_{b_1-y} [e^{-(q+r)\tau_0^-} ]$.
On the other hand, using the fact that $x \mapsto W^{(q+r)\prime}(x+)/W^{(q+r)}(x)$ is decreasing as in Remark 3.1(3) of \cite{HPY}, the mapping $x \mapsto \underline{r}^{(q+r)}(x+)/W^{(q+r)}(x)$ is also decreasing.
Therefore
\begin{align*}
\underline{r}^{(q+r)}((b-y)+)
&\leq \frac{W^{(q+r)}(b-y)}{W^{(q+r)}(b_1-y)}
\underline{r}^{(q+r)}((b_1-y)+)
\leq\frac{W^{(q+r)}(b_2-y)}{W^{(q+r)}(b_1-y)}
\underline{r}^{(q+r)}((b_1 - y)+).
\end{align*}
Because $W^{(q+r)}(b_2-y)/W^{(q+r)}(b_1-y)$ converges as $y \rightarrow -\infty$ by Remark \ref{remark_smoothness_zero}(3),
  for $-M$ small enough, there exists a constant $K(b_1,b_2)$ dependent only on $b_1,b_2$ 
  such that 
${W^{(q+r)}(b_2-y)} / {W^{(q+r)}(b_1-y)}\leq K(b_1,b_2)$ for all $y\leq -M$. 
Hence
\begin{align*}
\Big|
\frac{\partial}{\partial z} H^{(q+r)}(z, \theta) \big|_{z=(b-y)+}\Big|
\leq \theta \E_{b_1-y}\left[e^{-(q+r)\tau_0^-}\right]  + \left|\frac{\kappa(\theta)-(q+r)}{\theta-\Phi(q+r)}\right| K(b_1,b_2) \underline{r}^{(q+r)}((b_1 - y)+).
\end{align*}
Here by Lemma \ref{A3_fin} and
the polynomial growth of $f$ as in Assumption \ref{assump_f}(i), 
$\int_{-\infty}^{-M}| f(y) |\E_{b_1-y} [e^{-(q+r)\tau_0^-} ] \diff y <\infty$. 
For the second term we have, by the density function of $-\underline{X}(\mathbf{e}_{q+r})$ as in \eqref{density_running_min},
\begin{align*}
\int_{-\infty}^{-M}|f(y)|&
\underline{r}^{(q+r)}(b_1 - y)
\diff y =\int_{b_1+M}^{\infty}| f(b_1-u)|
\underline{r}^{(q+r)}(u)
\diff u \leq \frac {\Phi(q+r)} {q+r} \E\left[|f(b_1+\underline{X}(\mathbf{e}_{q+r}))| \right]<\infty,
\end{align*}
where the finiteness holds as in the proof of Lemma \ref{A3_fin}.
Hence, by Corollary 5.9 in \cite{Bartle}, the derivative can be interchanged over the integral and the proof is complete.

\subsection{Proof of Lemma \ref{smoothness}}\label{App_B}
(i) In view of the expression of  Lemma \ref{v_prime_matches}, by monotone convergence (noting that $f'$ is monotone) and \eqref{U_bound}, $v_{b^*}'$ is continuous for all $x\in\R$.

Therefore, it just remains to show that $v_{b^*}''$ is continuous for the case $X$ has paths of unbounded variation, where $W^{(q+r)}(0)=W^{(q)}(0)=0$ by Remark \ref{remark_smoothness_zero}(2).

Using the expression of Lemma \ref{v_prime_matches} together with Theorem \ref{theorem_resolvents}, we obtain after differentiation that 
\begin{align*}
v_{b^*}''(x) &= \frac{\Phi(q) } {r}(\Phi(q+r)-\Phi(q))Z^{(q)}(x,\Phi(q+r))\int_{-\infty}^{\infty}f'(y)H^{(q+r)}(b-y,\Phi(q))\diff y\\
&-\int_{b^*}^xf'(y)W^{(q)\prime}(x-y)\diff y- \frac \partial {\partial x}\int_{-\infty}^{b^*} f'(y)\Upsilon(x-b^*,y-b^*)\diff y.
\end{align*}
Because $\Upsilon(x-b^*,y-b^*)$ is continuous for the case of unbounded variation and by Lemma \ref{A2}, 
\begin{align*}
	\frac \partial {\partial x} \int_{-\infty}^{b^*} f'(y)  \Upsilon(x-b^*,y-b^*)
	\diff y = \int_{-\infty}^{b^*} f'(y)  \frac \partial {\partial x}  \Upsilon(x-b^*,y-b^*)
	\diff y, \quad x \in \R. 
\end{align*}
Using \eqref{common_derivative}, we can write, for $x \neq y$,
\begin{align*}
\frac \partial {\partial x}  \Upsilon(x-b^*,y-b^*)=A(x,y,b^*)-r\int_{b^*}^xW^{(q)}(x-z)A(z,y,b^*) \diff z,
\end{align*}
where
\begin{align*}
A(x,y,b^*)
:= \left(W^{(q+r)\prime}(x-y)-\Phi(q+r)W^{(q+r)}(x-y)\right) -\Phi(q+r) \Theta^{(q+r)}(x-b^*, y-b^*).
\end{align*}
(i) 
Because $A(x,y,b^*) = W^{(q)\prime}(x-y)$ for $y > b^*$,
\begin{align*}
K(x,b^*) := \int_{b^*}^xf'(y)W^{(q)\prime}(x-y)\diff y + \int_{-\infty}^{b^*}f'(y)A(x,y,b^*)\diff y =\int_{-\infty}^{b^* \vee x}f'(y)A(x,y,b^*)\diff y.
\end{align*}
For $x\leq b^*$, recalling that $W^{(q+r)}(0)=0$ as in Remark \ref{remark_smoothness_zero}(2) for the case of unbounded variation,
\begin{align*}
K(x,b^*)
=\frac{\Phi(q+r)}{q+r}\E\left[f'(\underline{X}(\mathbf{e}_{q+r})+x)\right] -\Phi(q+r)\e_x\Big[\int_0^{\tau_{b^*}^+}e^{-(q+r)t}f'(X(t))\diff t \Big].
\end{align*}
Similarly, for $x > b^*$, by Remark \ref{remark_theta_sign}(ii), 
\begin{align*}
	K(x,b^*)
	&=\frac{\Phi(q+r)}{q+r} \E\left[f'(\underline{X}(\mathbf{e}_{q+r})+x)\right]+ \Phi(q+r)e^{\Phi(q+r)(x-b^*)}\e_{b^*}\Big[\int_0^{\tau_x^+}e^{-(q+r)t}f'(X(t))\diff t\Big].
\end{align*}
(1) The function  $x \mapsto \E\left[f'(\underline{X}(\mathbf{e}_{q+r})+x)\right]$ is continuous by monotone convergence in view of Assumption \ref{assump_f}(i).
(2) By Assumption \ref{assump_f}(i), for $\underline{x} \leq x \leq \overline{x}$, under $\p$, 
\begin{align*}
\int_0^{\tau_{b^*-x}^+}e^{-(q+r)t} |f'(X(t) +x)|\diff t \leq \int_0^\infty e^{-(q+r)t} ( |f'(X(t) + \underline{x})| + |f'(X(t) + \overline{x})| )\diff t 
\end{align*}
which are integrable by Remark \ref{remark_finiteness_resolvents}. Hence, by dominated convergence, $x \mapsto \e_x [\int_0^{\tau_{b^*}^+}e^{-(q+r)t}f'(X(t))\diff t ]$ is continuous.  
(3) The function $x \mapsto \e_{b^*} [\int_0^{\tau_x^+}e^{-(q+r)t}f'(X(t))\diff t ]$ is continuous by again dominated convergence because the absolute value of the integrand is dominated by $\int_0^\infty e^{-(q+r)t}|f'(X(t))|\diff t$. In sum, $K(x,b^*)$ is continuous in $x$.

(ii) For the case $x>b^*$, we have by Fubini's theorem that
\begin{align}\label{n_1}
\int_{-\infty}^{b^*}f'(y)\int_{b^*}^xW^{(q)}(x-z)A(z,y,b^*)\diff z\diff y=\int_{b^*}^xW^{(q)}(x-z)\int_{-\infty}^{b^*}f'(y)A(z,y,b^*)\diff y\diff z.
\end{align}
Here, for $\underline{x} \leq x \leq \overline{x}$, 
\begin{align*}
W^{(q)}&(x-z)\left|\int_{-\infty}^{b^*}f'(y)A(z,y,b^*)\diff y\right|\leq W^{(q)}(\overline{x})\frac{\Phi(q+r)}{q+r}\Bigg[\E\big[|f'(\underline{X}(\mathbf{e}_{q+r})+\underline{x})| + |f'(\underline{X}(\mathbf{e}_{q+r})+\overline{x})|\big]\\
&+(q+r)e^{\Phi(q+r)(\overline{x}-b^*)} \E\left[\int_0^\infty e^{-(q+r)t} \big( |f'(X(t) + \underline{x})|\ +  |f'(X(t) + \overline{x})| \big) \diff t \right]\Bigg].
\end{align*}
Hence, by bounded convergence, the term defined in \eqref{n_1} is also continuous in $x$. This concludes the proof of the continuity of $v''_{b^*}$.

\end{document}